\documentclass[11pt,a4paper]{article}

\usepackage{a4wide}

\usepackage{amsmath}
\usepackage{amssymb}
\usepackage[utf8]{inputenc}
\usepackage{latexsym}
\usepackage{graphicx}
\usepackage{subcaption}
\usepackage{pgf,tikz}
\usetikzlibrary{backgrounds}
\usepackage{bbold}
\usetikzlibrary{arrows}

\usepackage{hyperref}
\usepackage{float}
\usepackage{color}
\usepackage{enumitem}
\usepackage[normalem]{ulem}
\usepackage[title]{appendix}

\newtheorem{theorem}{Theorem}[section]
\newtheorem{lemma}{Lemma}[section]
\newtheorem{corollary}{Corollary}[section]
\newtheorem{prop}{Proposition}[section]
\newtheorem{definition}{Definition}[section]
\newtheorem{remark}{Remark}

\definecolor{ddgreen}{rgb}{0,0.4,0.4}
\definecolor{dgreen}{rgb}{0,0.65,0}

\newenvironment{proof}{\noindent{\textsc{Proof.}}}
{$\hfill\Box$\vspace{0.1 cm}\\}
\newenvironment{proofof}[1]{\noindent{\textsc{Proof of #1.}}}
{$\hfill\Box$\vspace{0.1 cm}\\}

\newcommand{\R}{\mathbb R}

\newcommand{\N}{{\mathbb N}}

\newcommand{\F}{{\mathcal F}}
\newcommand{\NF}{{\mathcal NF}}

\newcommand{\tv}{\mathrm{TV}}
\newcommand{\tvp}{\mathrm{TV}^+}

\newcommand{\LL}[1]{\mathbf{L^{\pmb{#1}}}}
\newcommand{\Ll}[1]{\mathbf{L_{loc}^{\pmb{#1}}}}
\newcommand{\C}[1]{{\mathbf C^{\pmb{#1}}}}
\newcommand{\Cc}[1]{{\mathbf {C_c^{\pmb{#1}}}}}

\newcommand{\Wl}[1]{{\mathbf{W_{loc}^{\pmb{#1}}}}}

\newcommand{\eps}{\varepsilon}

\newcommand{\norm}[1]{\left\| #1 \right\|}
\newcommand{\abs}[1]{\left\vert #1 \right\vert}

\newcommand{\jumps}{\mathcal{J}}
\newcommand{\shocks}{\mathcal{S}}
\newcommand{\raref}{\mathcal{R}}
\newcommand{\pt}{\partial_t}
\newcommand{\px}{\partial_x}

\newcommand{\fhi}{\varphi}

\DeclareMathOperator{\sgn}{sgn}

\makeatletter
\let\@fnsymbol\@arabic
\makeatother

\title{A multiscale model for traffic regulation via autonomous vehicles}

\author{
\textsc{Mauro Garavello\footnotemark[1]}
\and
\textsc{Paola Goatin\footnotemark[2]}
\and
\textsc{Thibault Liard\footnotemark[3]}
\and
\textsc{Benedetto Piccoli\footnotemark[4]}
}

\date{\today}

\begin{document}

\footnotetext[1]{
  Department of Mathematics and its Applications,
  University of Milano - Bicocca,
  via R. Cozzi 55,
  20125 Milano, Italy. 
  E-mail: \texttt{mauro.garavello@unimib.it}
  }
  
\footnotetext[2]{
  Inria Sophia Antipolis - M\'editerran\'ee, 
  Universit\'e C\^ote d'Azur, Inria, CNRS, LJAD, 
  2004 route des Lucioles - BP 93, 06902 Sophia Antipolis Cedex, France.
  E-mail: \texttt{paola.goatin@inria.fr}
  }
  
\footnotetext[3]{
Chair of Computational Mathematics, Fundaci\'on Deusto, Av. de las Universidades 24, 48007 Bilbao, Basque Country, Spain. E-mail: \texttt{thibault.liard@deusto.es}
}

\footnotetext[4]{
  Department of Mathematical Sciences and
  Center for Computational and Integrative Biology,
  Rutgers University - Camden,
  311 N 5th Street, Camden, NJ 08102, USA.
  E-mail:  \texttt{piccoli@camden.rutgers.edu}
  }

\maketitle

\begin{abstract}
Autonomous vehicles (AVs) allow new ways of regulating the traffic flow on road networks.
Most of available results in this direction are based on microscopic approaches, where ODEs describe the evolution of regular cars and AVs.
In this paper, we propose a multiscale approach, based on recently developed models for moving bottlenecks.
{Our main result is the proof of existence of solutions for time-varying bottleneck speed,
which corresponds to
open-loop controls with bounded variation.}
\end{abstract}

\textit{Key Words:} Conservation laws; PDE-ODE system; macroscopic traffic models; moving bottlenecks; autonomous vehicles; control problems.

\textit{AMS Subject Classifications:} 90B20, 35L65.

\section{Introduction}\label{se:introduction}
Autonomous Vehicles (briefly AVs) 
represent the most disruptive technology for traffic regulation
\cite{harper2016cost,milakis2017policy,wadud2016help,wan2016optimal,
  wang2017comparing}. 
The effect of AVs in terms of influencing bulk traffic has been
studied \textsl{in-silico}
\cite{davis2004effect,gueriau2016assess,knorr2012reducing,
  talebpour2016influence,wang2016cooperative}, artificial environment
\cite{Jang:2019:SSC:3302509.3313784} 
and also in experiments \cite{STERN2018205}.
In particular, the results
of \cite{STERN2018205} showed a potential
decrease of up to 40\% in fuel consumption
by dampening of traffic waves.
Despite such achievements, a complete macroscopic theory for control of
bulk traffic via AVs is still missing.
The need of a macroscopic theory is due
to the curse of dimensionality 
preventing control design for microscopic
models~\cite{DelleMonache2019}.

Our approach to bypass current limitations
is based on the idea of using macroscopic models for the bulk traffic, consisting
of partial differential equations (PDEs), paired with microscopic ones for the AVs,
consisting of ordinary differential equations (ODEs).
The history of modeling traffic via partial differential equations started with the celebrated Lighthill-Whitham-Richards (briefly LWR)
model for traffic flow on a road \cite{Lighthill-Whitham_1955, Richards_1956}.
The authors assumed that the average speed $v$ depends only on local density, thus obtaining
a scalar conservation law.
More recently, the LWR model was paired
to an ODE via a moving flux constraint
producing a multiscale model, see
\cite{DMG2014,MR3264413,MR3657113}. Moreover,
other works used coupled ODE-PDE
systems~\cite{BCG2012, DMG2014, LattanzioMauriziPiccoli, LLG98, VGC2017}.

{Here we extend the multiscale model
of \cite{MR3264413} by assuming that
the AV desired speed, described by a control $u$,
is time-dependent, as in the case of regulation by a centralized controller.}
The resulting system is obtained by modifying the ODE for the AV, which
is given by the minimum between the
desired speed $u$ and the average speed $v$
of local downstream traffic. 
The proposed approach has already been investigated numerically
in~\cite{PGF2017}, using Model Predictive Control
and proving its effectiveness 
in reducing traffic nuisance.

We base our work on the many results
available for the multiscale model of \cite{MR3264413}. In particular,
existence and well-posedness of solutions
was investigated in~\cite{MR3657113, LiardPiccoli2, LiardPiccoli}
and the numerical
aspect in~\cite{CDMG2017, DMG2014, DMG2016}.
{ Moreover, the multiscale model was
  extended to second order in~\cite{VGC2017}, replacing the LWR  by Aw-Rascle-Zhang model.
}

The complete multiscale model is given by the
system~\eqref{eq:control-model}, consisting of a scalar conservation law,
a controlled ODE and a moving flux constraint (plus initial conditions).
To couple the equations, we need to specify the capacity reduction function
$F_\alpha$ for every AV desired speed. This gives rise to undercompressive shocks
also depending on the desired speed $u$, see~\eqref{eq:check-hat}.
The definition of solution needs to 
reflect this choice via
a modified entropy condition, see point~\ref{def:sol-CP:3}
of Definition~\ref{def:sol-CP}.\\
{Our main result is the proof
of existence of solutions for time-varying
speed profile of the AV with bounded variation. The latter corresponds for a centralized
controller to an open loop control $u$
with bounded variation in time.}
In order to achieve the result,
we construct approximate
solutions via wave-front tracking
approximations. 
In particular, we construct
approximation grids allowing all left and right densities of undercompressive
shocks for the discretized desired speed.\\
As usual, the sequence of approximate solutions satisfy some compactness
estimate, that allows to pass to the limit. In our case, we need to define a
Glimm type functional accounting for the density total variation, a term
for variation due to undercompressive shocks and the time-variation of the
{time-varying speed.}
We can prove that at each wave interaction such functional
either decreases (by a quantity bounded away from zero) or remains constant,
but in this case the number of waves does not increase. Moreover, a detailed
analysis is necessary to prove the convergence of the AV speed.
Once a limit of approximate solutions is obtained, one has to prove that it satisfies the various
conditions in the definition of solution. For this we use various results from
the above mentioned literature adapted to our case.
{We remark that these results can be extended
  to the case of a finite number of autonomous vehicles,
  provided their trajectories do not interact. This extension
  is straightforward since all the waves generated by the multiscale system propagate with finite speed.
}

The paper is organized as follows. Section~\ref{se:control_model} describes the
coupled PDE-ODE model and introduces the basic notation.
Section~\ref{se:cauchy} deals with the Cauchy problem. More precisely,
in Subsection~\ref{sse:wft} we describe the wave-front tracking algorithm, we construct
the approximating grids, and we define the Glimm type functional;
in Subsection~\ref{sse:interactions} we study the changes in the Glimm type functional,
due to wave interactions; finally
in Subsection~\ref{sse:existence-solution} we state and prove
the main theorem.
The paper ends with
Appendix~\ref{sec:app:tech-lemmas}, which contains various technical lemmas
used in Subsection~\ref{sse:existence-solution}, and with
 Appendix~\ref{se:RP}, which describes in details the solution to the Riemann
 problem. 

%
%
%
%
\section{Description of the control model}
\label{se:control_model}

Consider a unidirectional road $I$, model by the real line $\R$,
where the traffic is described by the LWR
model~\cite{Lighthill-Whitham_1955, Richards_1956}
\begin{equation}
  \label{eq:LWR-model}
  \pt \rho + \px f\left(\rho\right) = 0,
\end{equation}
where $\rho = \rho(t,x) \in [0, R]$ denotes the macroscopic traffic density
at time $t \ge 0$ and at position $x \in \R$, and $f = f(\rho)$
is the flux, which depends only on $\rho$.
The constant $R$ denotes the maximum possible density of the road.
As usual the flux $f$ is given by $\rho v\left(\rho\right)$, where
$v \in \C2\left([0, R]; [0, +\infty)\right)$
is the average speed of cars. We assume that
the flux satisfies the condition
\begin{enumerate}[label=\textbf{(F)}]
\item \label{hyp:F}
  $f: \C2\left([0, R]; [0, +\infty)\right)$,\quad $f(0) = f(R) = 0$, 
  \\	
  $f$ strictly concave: $-B\leq f''(\rho)\leq-\beta<0$
  for all $\rho\in [0,R]$, for some $\beta, B >0$. \\
  Moreover, we assume $v'(\rho)< 0$ for every $\rho\in\,]0,R[$.
\end{enumerate}
Note that, by~\ref{hyp:F}, the speed $v$ is a strictly decreasing function.

Assume that a vehicle, e.g. a police or autonomous vehicle (AV),
whose position is described
by the variable $y = y(t)$, aims at controlling the behavior of traffic,
selecting its maximal speed. Hence the evolution of such a vehicle
is described by
the ordinary differential equation (ODE)
\begin{equation*}
  \dot y(t) = \omega(\rho(t,y(t));u):=\min\{ u(t), v(\rho(t,y(t)_+))\},
\end{equation*}
where $u = u(t) \in [0, V]$, with $V:=\max_{\rho\in[0,R]} v(\rho)$,
is a control function, selecting the \textit{desired speed}
of the vehicle. Indeed, the vehicle can move at its desired speed
as long as the 
downstream traffic moves faster, otherwise it has to adapt.
{The notations $\rho\left(t, x_-\right)$ and
  $\rho\left(t, x_+\right)$
  stand respectively for the left and right trace,
  with respect to the variable $x$, of $\rho$ at the point $(t, x)$,
  i.e.
  \begin{equation*}
    \rho(t, x_-) := \lim_{\xi \to x^-} \rho(t, \xi)
    \qquad \textrm{ and } \qquad
    \rho(t, x_+) := \lim_{\xi \to x^+} \rho(t, \xi),
  \end{equation*}
  whose existence is ensured provided $\rho(t,\cdot)$ has finite
  total variation.
  In the sequel, the notation is used also for the
  trace with respect to the time variable $t$.}

Following the model proposed in~\cite{MR3264413, LLG98},
we consider the following system
\begin{subequations}
  \label{eq:control-model}
  \begin{align}
      &\pt \rho \left(t,x\right) + \px f\left(\rho\left(t,x\right)\right) = 0,
      &
        t > 0,\, x \in \R, \label{eq:LWR}
      \\[5pt]
      &\dot y(t) = \min\{ u(t), v(\rho(t,y(t)_+))\},
      &
        t > 0,
       \\[5pt]
    & f \left(\rho\left(t,y(t)\right)\right) 
      - \dot y(t) \rho \left(t,y(t)\right) \le
      F_\alpha\left(\dot y(t)\right):=\max_{\rho\in[0,R]}
      \left( \alpha f(\rho/\alpha)-\rho\dot y (t)\right),
      \label{eq:constraint}
      &
        t> 0,
       \\
      &\rho(0, x) = \rho_0(x),
      &
        x \in \R, \label{eq:initial}
      \\[5pt]
      &y(0) = y_0.&
  \end{align}
\end{subequations}
Above, $\rho_0$ and $y_0$ are the initial traffic density and AV position,
while the function $F_\alpha$ in~\eqref{eq:constraint}, $\alpha \in\ ]0, 1[$, represents the
road capacity reduction due to the presence of the AV, acting as a moving bottleneck which imposes
a unilateral flux constraint at the AV position.
To determine the function $F_\alpha$, we consider reduced flux function
\begin{equation*}
  \begin{array}{rccc}
    f_{\alpha}: 
    & [0, \alpha R] 
    & \longrightarrow 
    & \R^+
    \\
    & \rho 
    & \longmapsto 
    & 
      \rho v({\rho}/{\alpha})= \alpha f(\rho/\alpha),
  \end{array}
\end{equation*}
which is a strictly concave function satisfying
$f_{\alpha}(0) = f_{\alpha}(\alpha R) = 0$.
For every $u \in [0, V]$, define the point
$\tilde \rho_{u}$ as the unique solution to the equation
$f'_{\alpha}(\rho) = u$. Introduce also, for every $u \in [0, V]$,
the function
\begin{equation*}
  \begin{array}{rccc}
    \fhi_{u}: 
    & [0, R] 
    & \longrightarrow 
    & \R^+
    \\
    & \rho 
    & \longmapsto 
    & 
      f_{\alpha}(\tilde \rho_u) + u \left(\rho - \tilde \rho_u\right).
  \end{array}
\end{equation*}
Hence, if $\dot y(t)=u$, the function $F_\alpha$ in~\eqref{eq:constraint} is defined by
\begin{equation*}
  \begin{array}{rccc}
    F_\alpha: 
    & [0, V] 
    & \longrightarrow 
    & \R^+
    \\
    & u
    & \longmapsto 
    & \fhi_u(0) = f_{\alpha}(\tilde \rho_u) - u \tilde \rho_u.
  \end{array}
\end{equation*}
If $\dot y(t)=v(\rho(t,y(t)_+))$, then the inequality~\eqref{eq:constraint}
is trivially satisfied since the left-hand side is zero.
Finally, the points 
$0 \le \check \rho_u \le \tilde \rho_u \le \hat \rho_u \le \rho_u^* \le R$ are
uniquely defined by
\begin{equation}
  \label{eq:check-hat}
  \check \rho_u = \min \mathcal I_u, \quad
  \hat \rho_u = \max \mathcal I_u, \quad
  \mathcal I_u = \left\{\rho \in [0, R]:\, f(\rho) = \fhi_u(\rho)\right\},
\end{equation}
and implicitly by
\begin{equation}
  \label{eq:rho_u}
  v(\rho_u^*) = u,
\end{equation}
see~\cite{MR3264413} and~Figure~\ref{fig:check-hat}.
It is straightforward to see that $\check \rho_V = \tilde \rho_V = \hat \rho_V = \rho^*_V =0$. 

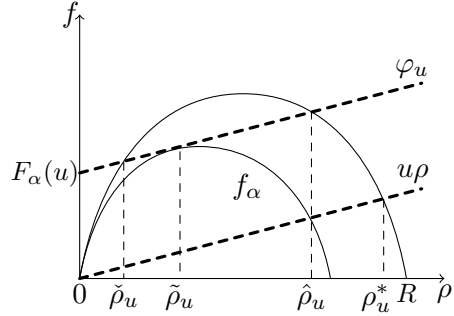
\begin{figure}
  \centering
    \begin{tikzpicture}[line cap=round,line join=round,x=1.cm,y=0.7cm]

      \draw[<->] (0.2, 6.) -- (0.2, 1.) -- (5, 1.);

      \draw (0.2, 1.) to [out = 80, in = 180] (2.35, 4.5)
      to [out = 0, in = 100] (4.5, 1.);

      \draw (0.2, 1.) to [out = 80, in = 180] (1.77, 3.5)
      to [out = 0, in = 100] (3.5, 1.);

      \draw[dashed, very thick] (0.2, 1.) -- (0.2 + 4.5, 1. + 1.7);
      \draw[dashed, very thick] (0.2, 3) -- (0.2 + 4.5, 3. + 1.7);

      \draw[dashed] (4.2, 2.5) -- (4.2, 1.);
      \draw[dashed] (3.25, 4.17) -- (3.25, 1.);
      \draw[dashed] (0.78, 3.2) -- (0.78, 1.);

      \draw[dashed] (1.52, 3.4) -- (1.52, 1.);

      \node[anchor=north,inner sep=0] at (0.2,0.9) {$0$};
      \node[anchor=east,inner sep=0] at (0.2,3) {\small $F_\alpha(u)$};
      \node[anchor=north,inner sep=0] at (4.1,0.9) {$\rho_u^*$};
      \node[anchor=north,inner sep=0] at (1.52,0.9) {$\tilde \rho_u$};
      \node[anchor=north,inner sep=0] at (3.25,0.9) {$\hat\rho_u$};
      \node[anchor=north,inner sep=0] at (0.78,0.9) {$\check\rho_u$};
      \node[anchor=north,inner sep=0] at (5,0.9) {$\rho$};
      \node[anchor=east,inner sep=0] at (0.2,6) {$f$};
      \node[anchor=north west,inner sep=0] at (2.2,3) {$f_{\alpha}$};
      \node[anchor=north,inner sep=0] at (4.5,0.9) {\small $R$};

      \node[anchor=south east,inner sep=0] at (4.8,4.8) {$\fhi_u$};
      \node[anchor=south east,inner sep=0] at (4.8,2.8) {$u \rho$};

    \end{tikzpicture}    
    \caption{The definition of $\tilde \rho_u$, $\check \rho_u$,
      $\hat \rho_u$ and $\rho_u^*$.}
  \label{fig:check-hat}
\end{figure}

\begin{lemma}
  \label{lem:monotonicity}
  Assume that the flux $f$ satisfies~\textup{\ref{hyp:F}}.  Let
  $\check\rho_u \leq \hat\rho_u\leq \rho^*_u$ be defined by~\eqref{eq:check-hat} and \eqref{eq:rho_u}.
  Then the maps $\check\rho(u):u\mapsto\check\rho_u$ and
  $\hat\rho(u):u\mapsto\hat\rho_u$ are strictly decreasing on $[0,V]$ and
  satisfy
  \begin{equation}
    \label{eq:check-hat-derivative}
    - \infty < \check{\rho}'(u) < 0
    \qquad
    - \infty < \hat{\rho}'(u) < 0
  \end{equation}
  for every $u \in \left[0, V\right[$.
  Moreover, the maps $\check\rho(u)$ and
  $\rho^*(u):u\mapsto \rho^*_u=v^{-1}(u)$ are also Lipschitz continuous
  functions.
  Finally we have
  \begin{equation}
    \label{eq:order-point-2}
    \check{\rho}_u < \tilde{\rho}_u < \frac{\tilde{\rho}_u}
    {\alpha} < \hat{\rho}_u <\rho^*_u
  \end{equation}
  for every $u \in \left[0, V\right[$.
\end{lemma}

\begin{proof}
  First note that $f'\left(\frac{\tilde{\rho}_u}{\alpha}\right) = u$;
  thus differentiating this expression with respect to $u$ we deduce
  that
  $f''\left(\frac{\tilde{\rho}_u}{\alpha}\right)
  \frac{\tilde{\rho}'_u}{\alpha} = 1$
  and so $\tilde{\rho}'_u < 0$ for every $u \in \, [0, V]$. This
  proves that $\tilde{\rho}_u$ is strictly decreasing with respect to
  $u$.

  Observe that $f_\alpha(\rho) < f(\rho)$ for every
  $\rho \in \, ]0, \alpha R]$ and that the function
  $H: [0, R] \to \R$, defined by $H(\rho) = f(\rho) - \fhi_u(\rho)$,
  is strictly concave.  By~\eqref{eq:check-hat},
  $H\left(\check{\rho}_u \right) = H\left(\check{\rho}_u \right) = 0$.
  Moreover
  $H\left(\tilde{\rho}_u \right) = f\left(\tilde{\rho}_u \right) -
  f_\alpha\left(\tilde{\rho}_u \right) > 0$
  for every $u \in [0, V[$. The concavity of $H$ implies that
  \begin{equation}
    \label{eq:order-point-1}
    \check{\rho}_u < \tilde{\rho}_u < \hat{\rho}_u<\rho^*_u
  \end{equation}
  for every $u \in [0, V[$. Besides, since the map
  $\rho\mapsto f(\rho)-u\rho$ is strictly concave for every fixed
  $u \in [0,V]$, with point of maximum at $\tilde\rho_u /\alpha$,
  by~\eqref{eq:check-hat} we also deduce that
  $\check{\rho}_u < \tilde{\rho}_u / \alpha < \hat{\rho}_u$.
  Therefore~\eqref{eq:order-point-1} implies~(\ref{eq:order-point-2}).

  By~\eqref{eq:check-hat}, we have
  \begin{equation}
    \label{eq:7-bis}
    f\left(\check{\rho}_u\right) - u \check{\rho}_u
    = f_\alpha \left(\tilde{\rho}_u\right) - u \tilde{\rho}_u
    \qquad \textrm{ and } \qquad
    f\left(\hat{\rho}_u\right) - u \hat{\rho}_u
    = f_\alpha \left(\tilde{\rho}_u\right) - u \tilde{\rho}_u.
  \end{equation}
  Differentiating equations~\eqref{eq:7-bis} w.r.t. $u$, for every
  $u \in [0, V[$ we have
  \begin{align}
    \label{eq:deriv-check} \check\rho'(u)
    &=
      \frac{\check{\rho}_u - \tilde{\rho}_u} {f'(\check{\rho}_u) -u} =
      \frac{\check\rho_u - \tilde\rho_u}{f''(\check\rho_u -\theta
      (\check\rho_u - \tilde\rho_u/\alpha) ) (\check\rho_u -
      \tilde\rho_u/\alpha)} \ge \frac{1}{f''(\check\rho_u -\theta
      (\check\rho_u - \tilde\rho_u/\alpha) )}\geq
      -\frac{1}{\beta},
    \\ 
    \label{eq:deriv-hat} \hat\rho'(u)
    &=
      \frac{\hat{\rho}_u - \tilde{\rho}_u}{f'(\hat{\rho}_u) -u} =
      \frac{\hat\rho_u - \tilde\rho_u}{f''(\hat\rho_u -\theta' (\hat\rho_u
      - \tilde\rho_u/\alpha) ) (\hat\rho_u - \tilde\rho_u/\alpha)} \le
      \frac{1}{f''(\hat\rho_u -\theta' (\hat\rho_u - \tilde\rho_u/\alpha)
      )} \leq -\frac{1}{B}
  \end{align}
  for some $\theta, \theta'\in [0,1]$.
  Using~\eqref{eq:order-point-2},~(\ref{eq:deriv-hat}), ~\ref{hyp:F},
  and the fact that $u = f'\left(\tilde{\rho}_u / \alpha\right)$, we
  deduce that~(\ref{eq:check-hat-derivative}) holds and that
  both $\check{\rho}_u$ and $\hat{\rho}_u$ are strictly
  decreasing as functions of $u\in [0,V[$.
  Finally, by~(\ref{eq:deriv-check}), we deduce that
  \begin{equation*}
    - \frac{1}{\beta} \le \check{\rho}'(u) < 0
  \end{equation*}
  for every $u \in [0, V[$, proving that the function
  $\check{\rho}(u)$ is Lipschitz continuous. By~\ref{hyp:F},
  $v$ is a  continuously differentiable  function and $v'(\rho)<0$
  for every $\rho \in \,]0, R[$.
  Using the Inverse Function Theorem, we deduce that $v^{-1}$
  is a  continuously differentiable function and $\rho^*(u)=v^{-1}(u)$.
  This concludes the proof.
\end{proof}

\begin{lemma}
  \label{le:grid_sequence}
  Let us consider the sequence $\omega_n$, defined by
  \begin{equation}
    \label{eq:grid_sequence}
    \omega_0 = 0, \qquad
    \omega_{n+1} = \hat{\rho}^{-1} \left(\check{\rho}\left(\omega_n\right)\right)
    \qquad \forall n \in \N. 
  \end{equation}
  Then $\omega_n$ is a strictly increasing sequence such that
  \begin{equation*}
    \lim_{n \to + \infty} \omega_n = V.
  \end{equation*}
\end{lemma}

\begin{proof}
  By Lemma~\ref{lem:monotonicity}, the functions
  \begin{equation*}
    \check \rho : \left[0, V\right] \to \left[0, \check{\rho}(0)\right],
    \qquad\qquad\qquad
    \hat \rho : \left[0, V\right] \to \left[0, \hat{\rho}(0)\right],
  \end{equation*}
  are strictly decreasing and so invertible. This implies that the sequence
  $\omega_n$, defined in~(\ref{eq:grid_sequence}), is well defined and
  $\omega_n \in \left[0, V\right]$ for every $n \in \N$.
  Note also that~(\ref{eq:check-hat-derivative})
  in Lemma~\ref{lem:monotonicity} implies that the inverse functions
  of $\check \rho$ and $\hat \rho$ are strictly decreasing functions.

  First we claim that $\omega_n < V$ for every $n \in \N$. We proceed by
  induction on $n$. When $n=0$, clearly $\omega_0 = 0 < V$. Assume now that
  $\omega_{n-1} < V$ and so
  $\check{\rho}\left(\omega_{n-1}\right) > 0 = \hat{\rho}\left(V\right)$.
  Hence $\hat{\rho}^{-1} \left(\check{\rho}\left(\omega_{n-1}\right)\right) < V$,
  i.e. $\omega_n < V$, proving the claim.

  We prove now that $\omega_n$ is a strictly increasing sequence.
  Fix $n \in \N$. Since $\omega_n < V$, we deduce, by~(\ref{eq:order-point-2}),
  that $\check{\rho}\left(\omega_n\right) < \hat{\rho}\left(\omega_n\right)$
  and so
  $\hat{\rho}^{-1}\left(\check{\rho}\left(\omega_n\right)\right) > \omega_n$,
  i.e. $\omega_{n+1} > \omega_{n}$.

  Since $\omega_n$ is an increasing sequence, then it has a limit $L$.
  Clearly $L \in [0, V]$. By~(\ref{eq:grid_sequence}), we deduce that
  \begin{equation*}
    L = \hat{\rho}^{-1}\left(\check{\rho}\left(L\right)\right)
    \quad\Longleftrightarrow \quad
    \hat{\rho} \left(L\right) = \check{\rho}\left(L\right)
    \quad\Longleftrightarrow\quad
    L = V,
  \end{equation*}
  concluding the proof.
\end{proof}


\begin{remark}
  \label{rem2}
  {\rm
    If $u_1 < u_2$, then either
    \begin{equation*}
      \check\rho(u_2) < \check\rho(u_1) < \hat\rho(u_2) < \hat\rho(u_1)
    \end{equation*}
    or
    \begin{equation*}
      \check\rho(u_2) \leq \hat\rho(u_2) < \check\rho(u_1)  < \hat\rho(u_1) .
    \end{equation*}
  }
\end{remark}

%
%
%
%
\section{The Cauchy problem}
\label{se:cauchy}

Let us consider an initial density
$\rho_0 \in \LL1 \left(\R; [0, R]\right)$ with finite total variation,
an initial position of the AV $y_0 \in \R$, and 
an open-loop control $u \in \LL1 \left(\R^+; [0, V]\right)$ with finite
total variation.
Following~\cite{AGS2010, ChalonsGoatinSeguin2013, MR3657113},
we define solutions to~\eqref{eq:control-model} as follows. 

\begin{definition}
  \label{def:sol-CP}
  The couple $\left(\rho, y\right)$ provides a solution
  to~\eqref{eq:control-model}
  if the following conditions hold.
  \begin{enumerate}
  \item \label{def:sol-CP:1}
    $\rho \in \C{0} \left(\R^+; \LL1 \left(\R; [0,R]\right)\right)$ and
    $\tv\left(\rho(t)\right) < + \infty$ for all $t \in \R^+$;
    

  \item \label{def:sol-CP:2} $y \in \mathbf{W^{1,1}_{loc}}(\R^+;\R)$;

  \item \label{def:sol-CP:3} 
For every $\kappa\in\R$ and for all $\varphi\in\Cc1 (\R^2;\R^+)$  it holds
\begin{align}
    \int_{\R^+}\int_\R &\left(
    |\rho-\kappa| \pt \varphi +\sgn (\rho-\kappa)(f(\rho)-f(\kappa)) \px\varphi
    \right) dx\, dt 
    +\int_\R |\rho_0-\kappa| \varphi(0,x) \, dx \nonumber \\
    & {+} \, 2 \int_{\R^+} \left(
    f(\kappa)  - \dot y (t) \kappa - \min\{f(\kappa) - \dot y (t) \kappa, F_\alpha(\dot y(t))\} 
    \right) \varphi(t,y(t)) \, dt \geq 0\,; \label{eq:EC}
\end{align}

  \item \label{def:sol-CP:5}
    For a.e. $t >0$, 
    $f \left(\rho\left(t,y(t)_\pm\right)\right) 
    - \dot y(t) \rho \left(t,y(t)_\pm\right)
    \le F_\alpha\left(\dot y(t)\right)$;
      
        \item \label{def:sol-CP:4}
    For a.e. $t > 0$, $\dot y(t) = 
    \min\left\{u(t), v\left(\rho\left(t, y(t)_+\right)\right)\right\}$.

  \end{enumerate}
\end{definition}

\begin{remark}{\rm
We observe that the definition introduced in~\cite{MR3264413}, requiring only the entropy admissibility on $\R^+ \times\, ]-\infty,y(t)[$ and  $\R^+ \times\, ]y(t),+\infty[$, is not sufficiently strong to single out a unique non-classical solution to Riemann problems. For example, the undercompressive shock between $\rho^*_u$ and $0$ traveling with speed $u$ would be an admissible solution in the sense of ~\cite[Definition 4.1]{MR3264413} for the initial datum $\rho_0(x)=\tilde\rho_u$. Instead, condition~\eqref{eq:EC}, as well as its equivalent formulation used in~\cite{MR3657113}, ensures that, if an undercompressive shock is present at $x=y(t)$, it satisfies point~\ref{def:sol-CP:5} in Definition~\ref{def:sol-CP} as an equality.
}
\end{remark}

%
%
\subsection{Wave-front tracking approximation}
\label{sse:wft}

Since solutions to Riemann problems are known analytically
(see details in Appendix~\ref{se:RP}),
we are able to construct
piecewise constant approximations of solutions to~\eqref{eq:control-model}
via the wave-front tracking algorithm,
see~\cite{Bressan_book_hyp_2000, MR1912206} for the general theory.

\begin{definition}\label{def:epswf}
  Given $\eps>0$, we say that the maps
  $ \rho_\eps$, $ y_\eps$, and $u_\eps$
  provide an $\eps$-approximate
  wave-front tracking solution to~\eqref{eq:control-model} if the following
  conditions hold. 
  \begin{enumerate}
  \item $\rho_{\eps}\in \C{0} \left(\R^+; \Ll1(\R; [0,R])\right)$ is
    piecewise constant, with discontinuities occurring along finitely many
    straight lines in the $(t,x)$-plane.
    Moreover jumps of $\rho_{\eps}$ can be shocks, rarefactions,
    or undercompressive shocks,
    and they are indexed by $\jumps(t) = \shocks(t) \cup \raref(t)
    \cup \mathcal{U}(t)$. For simplicity, the jumps of $\rho_\eps$
    are called \textbf{$\rho$-waves}.
    Moreover we call \textbf{classical waves} the $\rho$-waves described by the
    jump sets $\shocks$ and $\raref$.

  \item 
    $u_\eps \in \LL1\left(\R^+; \R^+\right)$
    is a piecewise
    constant function with a finite number of discontinuities.

  \item $y_\eps \in \mathbf{W^{1,1}_{loc}}\left(\R^+; \R\right)$ is a
    piecewise affine function. We refer to it as the \textbf{$y$-wave}.

  \item Along each shock
    $x(t)=x_{\alpha}(t)$, $\alpha\in\shocks(t)$, we have, for a.e. $t>0$,
    \begin{equation*}
      \rho_{\eps}(t,x_{\alpha}(t)_-)< \rho_{\eps}(t,x_{\alpha}(t)_+).
    \end{equation*}
    Moreover, for a.e. $t>0$,
    \begin{equation*}
      \abs{\dot x_{\alpha}(t)-
        \frac{f(\rho_{\eps}(t,x_{\alpha}(t)_-))
          -f(\rho_{\eps}(t,x_{\alpha}(t)_+))}
        {\rho_{\eps}(t,x_{\alpha}(t)_-)
          -\rho_{\eps}(t,x_{\alpha}(t)_+)}} \le \eps.
    \end{equation*}

  \item Along each rarefaction front
    $x(t)=x_{\alpha}(t)$, $\alpha\in\raref(t)$, we have, for a.e. $t>0$,
    \begin{equation*}
      \rho_{\eps}(t,x_{\alpha}(t)_+) < \rho_{\eps}(t,x_{\alpha}(t)_-)
      \leq \rho_{\eps}(t,x_{\alpha}(t)_+) + \eps.
    \end{equation*}
    Moreover, for a.e. $t>0$,
    \begin{equation*}
      \dot x_{\alpha}(t) \in \left[ f'(\rho_{\eps}(t,x_{\alpha}(t)_-)),
        f'(\rho_{\eps}(t,x_{\alpha}(t)_+)) \right] \,.
    \end{equation*}

  \item \label{point:NP}
    Along an undercompressive shock we have, for a.e. $t>0$,
    $x_\alpha(t)=y_\eps(t)$, $\alpha \in \mathcal U$, and
    \begin{equation*}
      \rho_{\eps}(t,y_\eps(t)_-) = \hat \rho_{u_\eps(t)}
      > \check \rho_{u_\eps(t)}
      = \rho_{\eps}(t,y_\eps(t)_+).
    \end{equation*}
    Moreover, for a.e. $t>0$,
    \begin{equation*}
      \dot x_{\alpha}(t) = {\dot y}_\eps(t)=u_\eps(t).
    \end{equation*}

  \item\label{point:initial} The following estimates hold
    \begin{equation*}
      \norm{\rho_{\eps}(0, \cdot) - \rho_{0}}_{L^1(\R)} < \eps,
      \qquad
      \norm{u_{\eps} - u}_{L^1(0, +\infty)} < \eps.
    \end{equation*}

  \item\label{point:7} For a.e. $t > 0$
    \begin{equation*}
      {\dot y}_\eps(t) = \min \left\{u_\eps (t),\, v \left(\rho_\eps 
        \left(t , {y}_\eps(t)_+\right)\right)\right\}.
    \end{equation*}
  \end{enumerate}
\end{definition}

\begin{remark}
  \label{rmk:wave_names}
  {\rm 
    With relation to Definition~\ref{def:epswf},
    we recall here the various types of waves.
    With the terms \emph{$\rho$-wave} and \emph{$y$-wave} we denote
    respectively a
    discontinuity for $\rho_\eps$ and the curve
    $t \mapsto \left(t, y_\eps(t)\right)$.
    \emph{Shocks}, \emph{rarefactions} and \emph{undercompressive shocks}
    are all $\rho$-waves. Moreover an \emph{undercompressive shock} is also a
    $y$-wave. Finally, by \emph{classical wave} we mean a $\rho$-wave, which is
    not a $y$-wave.
  }
\end{remark}

We describe here a possible algorithm for constructing a sequence of
approximate wave-front tracking solutions. First of all,
given $\nu \in \N$,
let us define the following grids on the interval $[0,R]$
for density and on $[0,V]$ for velocity.

\smallskip

{\bf Step 1.} Let $u_0^\nu=0$ and set recursively
\[
u_j^\nu > u_{j-1}^\nu \quad\hbox{such that} \quad \hat\rho_{u_j^\nu} = \check\rho_{u_{j-1}^\nu}
\quad\hbox{for} \quad j\geq 1.
\] 
and stop the procedure at $j=J_\nu$ such that $V-u_{J_\nu+1}^\nu < 2^{-\nu}$, see Figure~\ref{fig:gridA}.
By Lemma~\ref{le:grid_sequence},
the iterative procedure is finite.
We set $\rho_j^\nu = \check\rho_{u_{J_\nu-j}^\nu}$ for $j=0,\ldots, J_\nu$.

\smallskip

{\bf Step 2.} Divide the interval $[0, u_{1}^\nu]$ into $2^\nu$ parts, define the grid points
\[
u_{k,0}^\nu := k u_{1}^\nu 2^{-\nu}, \qquad k=0,\ldots, 2^\nu-1,
\]
the corresponding $\rho_{J_\nu+k}^\nu=\hat\rho_{u_{k,0}^\nu} \in \ ]\check\rho_0, \hat\rho_0[$
and consider the recursive sequence
\[
u_{k,i}^\nu \in \left] u_{i}^\nu,  u_{i+1}^\nu \right[ \quad\hbox{such that} \quad \hat\rho_{u_{k,i}^\nu} = \check\rho_{u_{k,i-1}^\nu}
\quad\hbox{for} \quad i=1,\ldots,J_\nu.
\]



\smallskip

{\bf Step 3.} To complete the mesh, we divide the remaining sub-interval 
$[\hat\rho_0,R]$
in $2^{\nu}$ parts, defining
\begin{align*}
\rho_\ell^\nu:= \hat\rho_0 + \ell (R-\hat\rho_0) 2^{-\nu}, &&\ell=1,\ldots, 2^{\nu}.
\end{align*}

\smallskip

{\bf Step 4.} Relabeling the points defined above, we obtain the grids $\mathcal{M}_\nu:=\left\{ \rho_i^\nu \right\}_{i=0}^{M_\nu}\subset [0,R]$ 
(with $\rho_0^\nu=0$ and $\rho_{M_\nu}^\nu=R$)
and $\mathcal{U}_\nu:=\left\{ u_i^\nu \right\}_{i=0}^{N_\nu}\subset [0,V]$
(with $u_0^\nu=0$ and $u_{N_\nu}^\nu=V$), 
where the points are labeled in increasing order. We also set 
\begin{align*}
&\delta_\rho^\nu:=\min_{i=1,\ldots,M_\nu} \left(\rho_{i}^\nu -\rho_{i-1}^\nu\right) \geq c_\nu 2^{-\nu}, 
&& \delta_u^\nu:=\min_{i=1,\ldots,M_\nu} \left(u_{i}^\nu -u_{i-1}^\nu\right) \geq c_\nu 2^{-\nu},\\
&\eps_\rho^\nu:=\max_{i=1,\ldots,M_\nu} \left(\rho_{i}^\nu -\rho_{i-1}^\nu \right)\leq C_\nu 2^{-\nu} ,
&& \eps_u^\nu:=\max_{i=1,\ldots,M_\nu} \left(u_{i}^\nu -u_{i-1}^\nu\right) \leq C_\nu 2^{-\nu},
\end{align*}
for some $c_\nu, C_\nu >0$ depending on $\alpha$, $f''$ and $J_\nu$.

\smallskip

Also, let $f_\nu:[0,R] \to [0,+\infty[$ be the piecewise linear function such that $f_\nu(\rho_i^\nu) = f(\rho_i^\nu)$ for each $\rho_i^\nu\in\mathcal{M}_\nu$, see Figure~\ref{fig:gridB}.

\begin{figure}
  \centering
  \begin{subfigure}[b]{0.45\textwidth}
\centering
    \begin{tikzpicture}
      \draw[<->] (0., 5.) -- (0., 0.) -- (5, 0.);
      \draw[very thick] (0., 4.5) to [out = -20,  in = 120] (4.5, 0.);
      \draw[very thick] (0., 3.) to [out = -60,  in = 160] (4.5, 0.);
      \draw (0., 3.) -- (2.4, 3.);
      \draw (2.4, 3.) -- (2.4, 0.);
      \draw (0., 0.8) -- (4., 0.8);
      \draw (4., 0.8) -- (4., 0.);
      \draw (0., 0.18) -- (4.4, 0.18);
      \draw[dashed] (0.5, 0.) -- (0.5, 4.3);
      \draw[dashed] (0., 4.3) -- (0.5, 4.3);
      \draw[dashed] (0., 2.25) -- (3.05, 2.25);
      \draw[dashed] (3.05, 0.) -- (3.05, 2.25);
      \draw[dashed] (1., 0.) -- (1., 4.05);
      \draw[dashed] (0., 4.05) -- (1., 4.05);
      \draw[dashed] (0., 1.75) -- (3.45, 1.75);
      \draw[dashed] (3.45, 0.) -- (3.45, 1.75);
      \draw[dashed] (1.5, 0.) -- (1.5, 3.75);
      \draw[dashed] (0., 3.75) -- (1.5, 3.75);
      \draw[dashed] (0., 1.35) -- (3.7, 1.35);
      \draw[dashed] (3.7, 0.) -- (3.7, 1.35);
      \draw[dashed] (2., 0.) -- (2., 3.35);
      \draw[dashed] (0., 3.35) -- (2., 3.35);
      \draw[dashed] (0., 1.05) -- (3.8, 1.05);
      \draw[dashed] (3.85, 0.) -- (3.85, 1.05);
      \draw[dashed] (0., 0.55) -- (4.1, 0.55);
      \foreach \Point in {(0.,4.5), (0.,3.), (0., 0.8), (0., 4.3), (0., 2.25), (0., 4.05), (0., 1.75), (0., 3.75), (0., 1.35), (0., 3.35), (0., 1.05)} {\node at \Point {\textbullet};}
      \foreach \Point in {(0.,0.), (4.5, 0.), (2.4, 0.), (4., 0.), (0.5, 0.), (3.05, 0.), (1., 0.), (3.45, 0.), (1.5, 0.), (3.7, 0.), (2., 0.), (3.85, 0.)} {\node at \Point {\textbullet};}
      \draw (-0.1, 0.) node[below] {$0$};
      \draw (2.4, 0.) node[below] {$u_1^\nu$};
      \draw (4., 0.) node[below] {$u_{2}^\nu$};
      \draw (4.6, 0.) node[below] {$V$};
      \draw (5., 0.) node[below] {$u$};
      \draw (0., 5.) node[left] {$\rho$};
      \draw (0., 4.5) node[left] {$\hat\rho_0$};
      \draw (0., 3.) node[left] {$\check\rho_0$};
      \draw (0., 0.8) node[left] {$\rho_{J_\nu-1}^\nu$};
      \draw (0., 0.18) node[left] {$\rho_{J_\nu-2}^\nu$};
      \draw (2.4, 3.1) node[right] {$\mathbf{\hat\rho(\cdot)}$};
      \draw (0.9, 1.) node[right] {$\mathbf{\check\rho(\cdot)}$};
      \draw (3., 4.5) -- (4., 4.5);
      \draw (4., 4.5) node[right] {Step 1};
      \draw[dashed] (3., 4.) -- (4., 4.);
      \draw (4., 4.) node[right] {Step 2};
    \end{tikzpicture}    
  \caption{Construction of the grids $\mathcal{M}_\nu$ and $\mathcal{U}_\nu$.}
  \label{fig:gridA}
      \end{subfigure} 
\hfill
\begin{subfigure}[b]{0.45\textwidth}
  \centering
   \begin{tikzpicture}
      \draw[<->] (0., 5.) -- (0., 0.) -- (5, 0.);
      \draw (0., 0.) to [out = 80,  in = 180] (2.25, 4.) to [out = 0,  in = 100] (4.5, 0.);
      \draw (0., 0.) to [out = 80,  in = 180] (1.6875, 3.) to [out = 0,  in = 100] (3.375, 0.);
      \draw (0.7, 3.) -- (3.8, 3.);
      \draw (0.7, 3.) -- (0.3, 1.6);
      \draw[dotted] (0.7, 0.) -- (0.7, 3.);
      \draw[dotted] (3.8, 0.) -- (3.8, 3.);
      \draw[dotted] (0.3, 0.) -- (0.3, 1.6);
      \draw[dashed] (0.35, 1.9) -- (1.4, 3.8);
      \draw[dashed] (0.45, 2.25) -- (2.1, 4);
      \draw[dashed] (0.55, 2.55) -- (2.8, 3.9);
      \draw[dashed] (0.6, 2.75) -- (3.4, 3.55);
      \draw[thick] (0., 0.) -- (0.3, 1.6) -- (0.35, 1.9) -- (0.45, 2.25) -- (0.55, 2.55) -- (0.6, 2.75) -- (0.7, 3.)
                          -- (1.4, 3.8) -- (2.1, 4) -- (2.8, 3.9) -- (3.4, 3.55) -- (3.8, 3.) -- (4.1, 2.1) -- (4.5, 0.);
      \draw (-0.1, 0.) node[below] {$0$};
      \draw (5., 0.) node[below] {$\rho$};
      \draw (0., 5.) node[left] {$f$};
      \draw (4.5, 0.) node[below] {$R$};
      \draw (3.3, 0.) node[below] {$\alpha R$};
      \draw (3.9, 0.) node[below] {$\hat\rho_0$};
      \draw (0.7, 0.) node[below] {$\check\rho_0$};
      \draw (0.3, 0.) node[below] {$\rho_j^\nu$};
    \end{tikzpicture}    
  \caption{Construction of $f_\nu$.}
  \label{fig:gridB}
      \end{subfigure}  
      \caption{Visual representation of the construction of the grids
        for $\rho_\nu$ and $u_\nu$ and of the construction of $f_\nu$.}  
\label{fig:grid}
\end{figure}
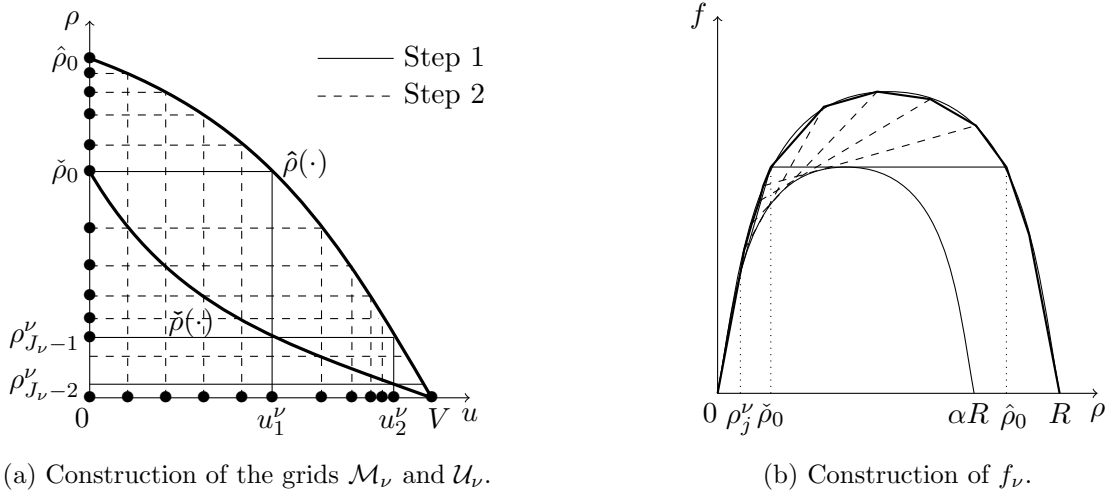

\medskip

Given an initial traffic density
$\rho_0 \in \LL1\left(\R; [0, R]\right)$ with finite total variation,
we consider a sequence of piece-wise constant functions
$\rho_{0,\nu}:\R\to\mathcal{M}_\nu$
such that $\rho_{0,\nu}$
has a finite number of discontinuities and
\begin{equation} 
  \label{eq:rho0approx}
  \lim_{\nu\to+\infty} \norm{\rho_{0,\nu} - \rho_{0}}_{\LL1\left(\R\right)} =0
  \qquad\hbox{and}\qquad
  \tv(\rho_{0,\nu}) \leq \tv(\rho_{0})
  \qquad\hbox{for all}~\nu\in\N.
\end{equation}
Besides, given $u \in \LL1\left(\R^+; [0, V]\right)$
with finite total variation,
we fix a sequence of piecewise constant functions
$u_\nu: \R^+\to \mathcal{U}_\nu$,
such that $u_\nu$ has a finite number of discontinuities and 
\begin{equation}
  \label{eq:uapprox}
  \lim_{\nu\to+\infty} \norm{u_\nu - u}_{\LL1 \left(\R^+\right)} =0
  \qquad\hbox{and}\qquad
  \tv(u_\nu) \leq \tv(u)
  \qquad\hbox{for all}~\nu\in\N.
\end{equation}
For every $\nu \in \N \setminus \{0\}$,
we apply the following procedure. At time $t=0$, we
solve all the (classical) Riemann problems determined by a discontinuity of $\rho_{0,\nu}$
and the constrained Riemann problem at the AV initial position $y_0$
(see~Appendix~\ref{se:RP}), replacing the function $f$ by $f_\nu$.
In this way, for small times $t>0$ we obtain a piecewise constant function $\rho_\nu=\rho_\nu(t,x)$ with values in $\mathcal{M}_\nu$
and whose jump discontinuities satisfy the Rankine-Hugoniot condition
\begin{equation*}
  \dot x_\alpha (t) = \frac{f_\nu(\rho_\nu(t,x_\alpha(t)_-))
    - f_\nu(\rho_\nu(t,x_\alpha(t)_+))}{\rho_\nu(t,x_\alpha(t)_-)
    - \rho_\nu(t,x_\alpha(t)_+) }
  = \frac{f(\rho_\nu(t,x_\alpha(t)_-))
    - f(\rho_\nu(t,x_\alpha(t)_+))}{\rho_\nu(t,x_\alpha(t)_-)
    -\rho_\nu(t,x_\alpha(t)_+) },
\end{equation*}
thus providing  a weak solution of~\eqref{eq:LWR}.
In particular, every rarefaction wave is approximated by a rarefaction fan, formed by
rarefaction shocks of strength less than $\eps_\rho^\nu$.
We repeat the previous construction at every time $\bar t$
at which the following possibilities occur:
\begin{enumerate}
\item two classical waves (shock or rarefaction jumps) of $\rho_\nu$ interact together;
  
\item a classical discontinuity of $\rho_\nu$ interacts with $y_\nu$;

\item $u_\nu(\bar t_-) \ne u_\nu(\bar t_+)$.
\end{enumerate}
In this way we construct a piecewise constant function $\rho_\nu$ and 
a piecewise linear function $y_\nu$.

\begin{remark}{\rm
    By slightly modifying the wave speeds, we may assume that, at every
    positive time $t$, at most one of the previous interactions happens.
  }
\end{remark}

\begin{remark}
  \label{rem:2}
  {\rm As usual, since rarefaction waves are generated only at time $t=0$
    or along the AV trajectory,
    we need to split rarefaction waves into rarefaction
    fans just at time $t=0$ and possibly at the discontinuity points for
    $u_\nu$. }
\end{remark}

Given a wave-front tracking approximate solution
$\left(\rho_\eps, y_\eps, u_\eps\right)$ to~\eqref{eq:control-model},
we define the Glimm type functional
\begin{equation} \label{eq:glimm}
  \Upsilon (t) = \Upsilon(\rho_\eps(t,\cdot),u_\eps)
  := \tv\left(\rho_\eps(t,\cdot)\right) + 2R + \gamma(t)
  + \frac{6}{\beta} \tv\left(u_\eps(\cdot); [t,+\infty[\right),
\end{equation}
where $\gamma$ is given by 
\begin{equation*}
  \gamma(t):=
  \begin{cases}
    -2\left(\hat\rho_{u_\eps(t)}-\check\rho_{u_\eps(t)}\right)
    & \hbox{if}~\rho_\eps(t,y_\eps(t)_-)=
    \hat\rho_{u_\eps(t)},~\rho_\eps(t,y_\eps(t)_+)= \check\rho_{u_\eps(t)},
    \\
    0
    &
    \hbox{otherwise}.
  \end{cases}
\end{equation*}
It is clear that $\Upsilon$ is well defined for a.e.
$t \geq 0$ and it changes only at discontinuity points of $\bar u_\eps$ or at interaction times. Moreover, we observe that
$\Upsilon(t)\geq \tv\left(\rho_\eps(t,\cdot)\right)  \geq 0$ and,
with the above choices \eqref{eq:rho0approx}, \eqref{eq:uapprox}, 
$\Upsilon(0)\leq \tv(\rho_0) + 2R + \frac{6}{\beta} \tv(u)$.
The functional $\Upsilon$ will serve to provide an uniform estimate on the total variation
of the approximate solutions $\rho_\nu$ constructed above.

%
%
\subsection{Interaction estimates}
\label{sse:interactions}

This subsection we will show the following property of the functional $\Upsilon$.
\begin{prop} \label{prop:upsilon}
Let $\left\{\rho_\nu,y_\nu,u_\nu\right\}_{\nu\in\N}$ be the sequence of approximate solutions constructed in Section~\ref{sse:wft}.
For any $\nu\in\N$, at any wave interaction or jump in $u_\nu$ the map $t\mapsto\Upsilon(t)=\Upsilon(\rho_\nu(t,\cdot),u_\nu)$
either decreases by at least $\min\left\{ \delta_\rho^\nu, 6\delta_u^\nu / \beta \right\}$, or it remains constant and the number of waves does not increases.
\end{prop}

\begin{proof}
  We will detail the different types of interactions separately.
  To this end, we introduce the following notations:
  \begin{itemize}
  \item \emph{$\F_u$-wave}: a wave denoting the AV trajectory with
    maximum speed $u$ without discontinuity in $\rho$. The notation is
    to indicate a fictitious wave.

  \item \emph{$\NF_u$-wave}: a wave denoting the AV trajectory with
    maximum speed $u$ and discontinuity in $\rho$. The notation is to
    indicate a non fictitious wave. We can distinguish two cases:
    \begin{itemize}
    \item \emph{$UC_u$-wave}: a wave $(\rho_l,\rho_r)$ denoting the AV
      trajectory with maximum speed $u$ verifying $\rho_l=\hat \rho_u$
      and $\rho_r=\check \rho_u$. The notation indicates an
      undercompressive shock.

    \item \emph{$C_u$-wave}: a wave $(\rho_l,\rho_r)$ denoting the AV
      trajectory with speed $u$ verifying $\rho_l<\rho_r$. The
      notation indicates a classical shock.
    \end{itemize}
  \end{itemize}
  Let us consider an interaction occurring at time $t=\bar t$ away
  from the AV trajectory.  In this case, either two shocks collide, or
  a shock and a rarefaction front interact. In both cases the number
  of waves diminishes.  Moreover,
  $\tv(\rho_\nu(\bar t_+,\cdot)) \leq \tv(\rho_\nu(\bar t_-,\cdot))$,
  while the other terms in~\eqref{eq:glimm} remain constant and we
  conclude that $\Upsilon(\bar t _+)\leq \Upsilon(\bar t_-)$.

  We then focus on events involving the AV trajectory.  Interactions
  between classical waves and $y_\nu$ have been studied
  in~\cite{MR3264413}.  Since the functional used there is equivalent
  to the one defined in~\eqref{eq:glimm} when the control $u_\nu$ does
  not jump, we can conclude as in~\cite[Lemma 2]{MR3264413} that at
  any interactions of this type either $\Upsilon$ decreases of at
  least $\delta_\rho^\nu$, or remains constant and the number of waves
  does not increases.

  Therefore, we focus here on the situations in which a jump in $u$
  occurs.


\begin{lemma}
  \label{Y1}
  Assume that, at time $t = \bar t$, the control jumps from
  $u^-=u(\bar t_-)$ to $u^+=u(\bar t_+)$ and that
  we have a $\F_{u^-}$-wave at $t=\bar t_-$. 
  We have the following two cases.
  \begin{enumerate}
  \item At $\bar t_+$ we have a $\F_{u^+}$-wave and
    no wave is produced (see Figure~\ref{NSleftA}).

  \item At $\bar t_+$ we have a $UC_{u^+}$-wave and
    the number of waves increases (see Figure~\ref{NSleftB}).
  \end{enumerate}
  In both cases, we have the estimate
  \begin{equation*}
    \Delta \Upsilon(\bar t) = -\frac{6}{\beta } \abs{u^+-u^-} 
    \le - \frac{6}{\beta} \delta_u^\nu < 0.
  \end{equation*}
\end{lemma}

\begin{proof}
  Before time $\bar t$, we have a $\F_{u^-}$-wave and we denote the density
  by $\rho$. Thus, the speed of the AV is $w(\rho; u^-)$ and
  $\rho \not\in \left]\check \rho_{u^-},\hat \rho_{u^-}\right[$.
  At $t=\bar t$, a jump in the control from $u^-$ to $u^+$ occurs
  and two cases may happen:  
  \begin{enumerate}
  \item $\rho \not \in \left]\check \rho_{u^+}, \hat \rho_{u^+}\right[$:
    no new wave is produced (see Figure \ref{NSleftA}). Thus, we have: 
    \begin{align*} 
      & \Delta \tv(\rho(\bar t,\cdot))=0, \\
      & \Delta \gamma(\bar t)=0,\\
      & \Delta \tv(u;[\bar t,+\infty[)=-\abs{u^+-u^-}.
    \end{align*}
    We conclude that
    $\Delta \Upsilon(\bar t) = -\frac{6}{\beta }\abs{u^+-u^-} 
    \le -\frac{6}{\beta} \delta_u^\nu < 0 $.

\item $\rho \in \left]\check \rho_{u^+}, \hat \rho_{u^+}\right[$:
  a $UC_{u^+}$-wave arises together with two shocks
  $(\rho,\hat \rho_{u^+})$ and  $(\check \rho_{u^+},\rho)$
  (see Figure \ref{NSleftB}). Thus, we have: 
  \begin{align*} 
    & \Delta \tv(\rho(\bar t,\cdot))=\abs{ \rho-\check \rho_{u^+}}
      +\abs{\check \rho_{u^+}-\hat \rho_{u^+}}+\abs{\hat \rho_{u^+}-\rho}
      =2\abs{\check \rho_{u^+}-\hat \rho_{u^+}}, \\
    & \Delta \gamma(\bar t)=-2\abs{\check \rho_{u^+}-\hat \rho_{u^+}},\\
    & \Delta \tv(u;[\bar t,\infty))=-\abs{u^+-u^-}.
  \end{align*}
  We conclude that $\Delta \Upsilon(\bar t) = -\frac{6}{\beta} \abs{u^+-u^-}
  \le -\frac{6}{\beta} \delta_u^\nu < 0$.
 \end{enumerate}
 This concludes the proof.
\end{proof}
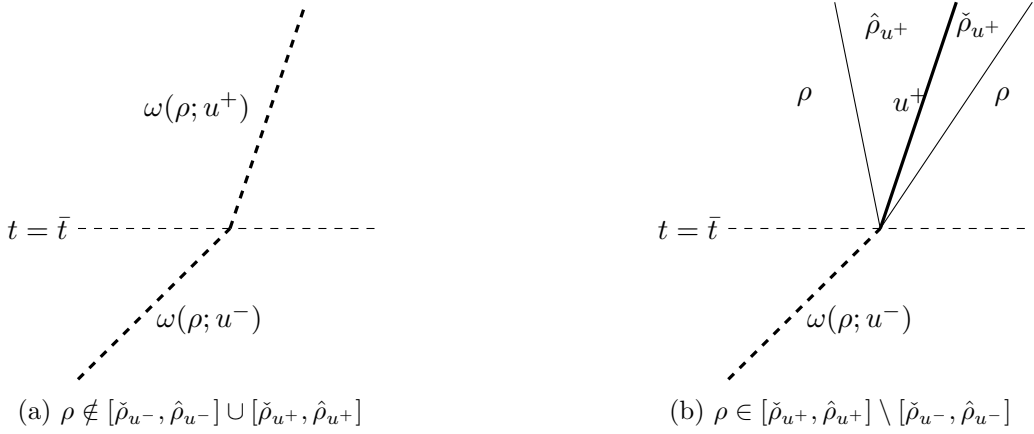
\begin{figure}[t]
\centering
\begin{subfigure}[b]{0.45\textwidth}
\centering
 \begin{tikzpicture}[scale=2]
\draw[dashed,very thick] (0,0) -- (1,1);
\draw[dashed,very thick] (1,1) -- (1.5,2.5);
\draw[dashed] (0,1) -- (2,1);
\draw (0,1) node[left] {$t=\bar t$};
\draw (0.45,0.4) node[right] {$\omega(\rho;u^-)$};
\draw (1.2,1.8) node[left] {$\omega(\rho;u^+)$};
\end{tikzpicture}
\caption{$ \rho \notin [\check \rho_{u^-},\hat \rho_{u^-}] \cup [\check \rho_{u^+},\hat \rho_{u^+}]$  }
\label{NSleftA}
\end{subfigure} 
\hfill
\begin{subfigure}[b]{0.45\textwidth}
\centering
\begin{tikzpicture}[scale=2]
\draw[dashed,very thick] (0,0) -- (1,1);
\draw[very thick] (1,1) -- (1.5,2.5);
\draw[dashed] (0,1) -- (2,1);
\draw (0,1) node[left] {$t=\bar t$};
\draw (0.45,0.4) node[right] {$\omega(\rho;u^-)$};
\draw (1.2,2) node[below] {$u^+$};
\draw (1,1) -- (0.7,2.5);
\draw (1,1) -- (2,2.5);
\draw (0.5,2) node[below] {$\rho$};
\draw (1.8,2) node[below] {$\rho$};
\draw (1.65,2.5) node[below] {$\check \rho_{u^+}$};
\draw (1.05,2.5) node[below] {$\hat \rho_{u^+}$};
\end{tikzpicture}
\caption{$ \rho \in [\check \rho_{u^+},\hat \rho_{u^+}]\setminus [\check \rho_{u^-},\hat \rho_{u^-}]$ }
\label{NSleftB}
\end{subfigure}
\caption{\label{NSleft} Jump from $u^-$ to $u^+$
  at $t=\bar t$  with a $\F_{u^-}$-wave at $t=\bar t_-$.}
\end{figure}


\begin{lemma}
  \label{Y5}
  Assume that, at time $t = \bar t$, the control jumps from
  $u^- = u \left(\bar t_-\right)$ to $u^+ = u \left(\bar t_+\right)$
  and that we have a $C_{u^-}$-wave 
  $(\rho_l,\rho_r)$ at $t = \bar t_-$.
  Then, at $t=\bar t_+$, we have a $\F_{u^+}$-wave and 
  no wave is produced (see Figure~\ref{NSleft3}). Moreover
  \begin{equation*}
    \Delta \Upsilon(\bar t) = -\frac{6}{\beta}\abs{ u^+-u^- }
    \le - \frac{6}{\beta} \delta_u^\nu <0.  
  \end{equation*}
\end{lemma}
\begin{proof}
  Since $(\rho_l,\rho_r)$ is a $C_{u^-}$-wave, it holds
  $\rho_l\leq \check\rho_{u^-}$ and $\rho_r\geq \hat\rho_{u^-}$.
  Therefore, by Lemma~\ref{lem:monotonicity}, we have
  $\rho_l\leq \check\rho_{u^+}$ if $u^+ < u^-$ and
  $\rho_r \geq \hat\rho_{u^+}$ if $u^+ > u^-$.  In both cases the
  classical shock $(\rho_l,\rho_r)$ satisfies
  the constraint~(\ref{eq:constraint}) at the position
  $y_\nu(\bar t)$ of the AV.
  Thus:
  \begin{align*}
    & \Delta \tv(\rho(\bar t,\cdot))=0, \\
    & \Delta \gamma(\bar t)=0,\\
    & \Delta \tv(u;[\bar t,+\infty[)=-\abs{u^+-u^- }.
  \end{align*}
  Therefore
  \begin{equation*}
    \Delta \Upsilon(\bar t) = -\frac{6}{\beta} \abs{u^+-u^- }
    \le -\frac{6}{\beta} \delta_u^\nu < 0,
  \end{equation*}
  concluding the proof.
\end{proof}
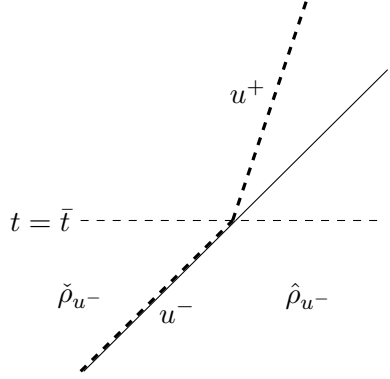
\begin{figure}[t!]
\centering
\begin{tikzpicture}[scale=2]
\draw[dashed,very thick] (0,0) -- (1,1);
\draw (1.02,1) -- (2.02,2);
\draw (0.02,0) -- (1.02,1);
\draw[dashed,very thick] (1,1) -- (1.5,2.5);
\draw[dashed] (0,1) -- (2,1);
\draw (0,1) node[left] {$t=\bar t$};
\draw (0.45,0.4) node[right] {$u^-$};
\draw (1.1,2) node[below] {$u^+$};
\draw (1.5,0.5) node {$\hat \rho_{u^-}$};
\draw (0,0.5) node {$\check \rho_{u^-}$};
\end{tikzpicture}
\caption{Jump from $u^-$ to $u^+$  ($u^+<u^-$ ) at $t=\bar t$
  with a $C_{u^-}$-wave at $t=\bar t_-$:
  $\rho_l \leq \check \rho_{u^-}$ and $\rho_r \geq \hat \rho_{u^-}$.} 
\label{NSleft3} 
\end{figure}


\begin{lemma}
  \label{Y2}
  Assume that, at time $t=\bar t$, the control jumps from $u^-$ to $u^+$
  and that we have a $UC_{u^-}$-wave at $t=\bar t_-$.
  The following two cases may happen.
  \begin{enumerate}
  \item The undercompressive shock is canceled.
    Then a rarefaction fan and a $\F_{u^+}$-wave
    arise (see Figure~\ref{NSleft2}).

  \item At time $t=\bar t_+$ we have a $UC_{u^+}$-wave (see Figure~\ref{2W}).
    Then, the number of waves increases, since a rarefaction fan and a shock
    wave are produced.
  \end{enumerate}
  In both cases we have $\Delta \Upsilon(\bar t) \leq -\delta_\rho^\nu$.
\end{lemma}

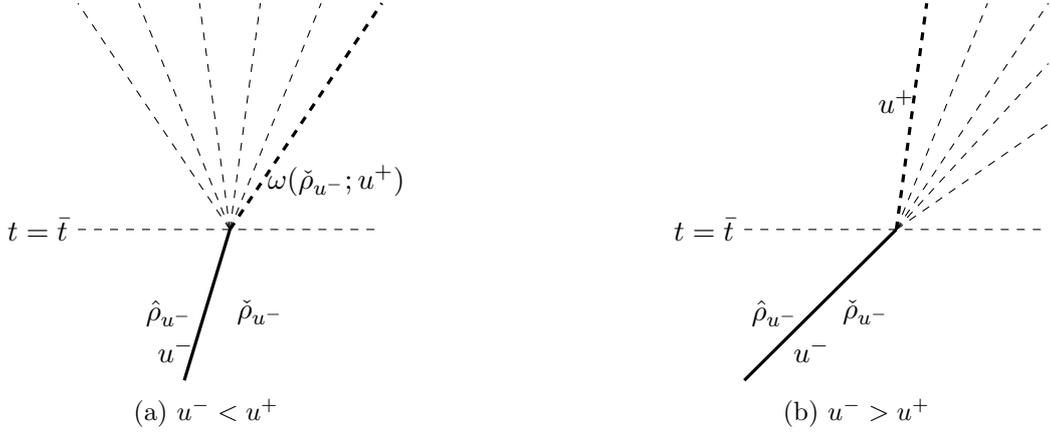
\begin{figure}[t!]
\centering
\begin{subfigure}[b]{0.45\textwidth}
\centering
      \begin{tikzpicture}[scale=2]
\draw[very thick] (0.7,0) -- (1,1);
\draw[dashed,very thick] (1,1) -- (2,2.5);
\draw[dashed] (0,1) -- (2,1);
\draw (0,1) node[left] {$t=\bar t$};
\draw (1.2,0.6) node[below] {$\check \rho_{u^-}$};
\draw (0.45,0.2) node[right,very thick] {$u^-$};
\draw (1.7,1.5) node[below,very thick] {$\omega(\check \rho_{u^-};u^+)$};
\draw[dashed] (1,1)--(0,2.5);
\draw[dashed] (1,1)--(0.4,2.5);
\draw[dashed] (1,1)--(0.8,2.5);
\draw[dashed] (1,1)--(1.2,2.5);
\draw[dashed] (1,1)--(1.6,2.5);
\draw (0.6,0.6) node[below] {$\hat \rho_{u^-}$};
\end{tikzpicture}
\caption{$u^- < u^+ $}
\label{NSleft2A}
\end{subfigure} 
\hfill
\begin{subfigure}[b]{0.45\textwidth}
\centering
     \begin{tikzpicture}[scale=2]
\draw[very thick] (0,0) -- (1,1);
\draw[dashed] (1,1) -- (2,2.5);
\draw[dashed] (0,1) -- (2,1);
\draw (0,1) node[left] {$t=\bar t$};
\draw (0.8,0.6) node[below] {$\check \rho_{u^-}$};
\draw (0.25,0.2) node[right,very thick] {$u^-$};
\draw (1,2) node[below,very thick] {$u^+$};
\draw[dashed, very thick] (1,1)--(1.2,2.5);
\draw[dashed] (1,1)--(1.6,2.5);
\draw[dashed] (1,1)--(2,2.1);
\draw[dashed] (1,1)--(2,1.7);
\draw (0.2,0.6) node[below] {$\hat \rho_{u^-}$};
\end{tikzpicture}
\caption{$u^- > u^+ $}
\label{NSleft2B}
\end{subfigure} 
\caption{Jump in $u$  at $t=\bar t$ with a $UC_{u^-}$-wave at $t=\bar t_-$
  which gets canceled at $t=\bar t_+$.}
   \label{NSleft2}
\end{figure}

\begin{proof}
  At $t=\bar t_-$, we have a $UC_{u^-}$-wave
  $(\hat \rho_{u^-},\check \rho_{u^-})$ and the speed of the AV is
  $u^-$. At $t=\bar t$, a jump from $u^-$ to $u^+$ occurs and two
  cases may happen.
  \begin{enumerate}
  \item The undercompressive shock disappears and so
    the Riemann problem at $t=\bar t$ is solved by a rarefaction fan.
    Hence we have
    \begin{align*}
      & \Delta \tv(\rho(\bar t,\cdot))=0, \\
      & \Delta \gamma(\bar t)=2 \left(\hat \rho_{u^-}-\check \rho_{u^-}\right),\\
      & \Delta \tv(u;[\bar t,+\infty[)=-\abs{u^+-u^-},
    \end{align*}
    which gives
    \begin{equation}
      \label{lk}
      \Delta \Upsilon(\bar t)=2\left( \hat \rho_{u^-}
        -\check \rho_{u^-}\right)-\frac{6}{\beta}\abs{u^+-u^-}.
    \end{equation}
    We distinguish two cases.
    \begin{enumerate}
    \item $u^- < u^+$ and $\hat{\rho}_{u^+} \le \check{\rho}_{u^-}$
      (see Figure~\ref{NSleft2A}). In this situation we have
      $f'(\check\rho_{u^-}) \leq u^+$ and
      $\check \rho_{u^+} \leq \hat \rho_{u^+} , \check \rho_{u^-} <
      \hat \rho_{u^-}$. By Taylor formula and~\ref{hyp:F}, there exists
      $\xi\in\ ]\check \rho_{u^-},\hat \rho_{u^-}[$ such that
      \begin{align}
        u^+-u^- 
        &\geq f'(\check \rho_{u^-})-\frac{f(\hat\rho_{u^-}) 
          - f(\check\rho_{u^-})}{\hat\rho_{u^-}-\check\rho_{u^-}}\nonumber\\
        &= -\frac{f''(\xi)}{2}(\hat \rho_{u^-}-\check \rho_{u^-})\label{lklk}\\
        &\geq \frac{\beta}{2}  (\hat \rho_{u^-}-\check \rho_{u^-}).\nonumber
      \end{align}
      From \eqref{lk} and \eqref{lklk} we deduce
      \begin{equation*}
        \Delta \Upsilon (\bar t)\leq - \left( \hat \rho_{u^-}-\check
          \rho_{u^-}\right) \leq -\delta_\rho^\nu.
      \end{equation*}
  
      \begin{figure}[t!]
        \centering
        \begin{subfigure}[b]{0.45\textwidth}
          \centering
          \begin{tikzpicture}[scale=2]
            \draw[very thick] (0.7,0) -- (1,1); \draw[very thick]
            (1,1) -- (1.9,2.5); \draw[dashed] (0,1) -- (2.2,1); \draw
            (0,1) node[left] {$t=\bar t$}; \draw (1.4,0.8) node[below]
            {$\check \rho_{u^-}$}; \draw (1.9,2) node[below]
            {$\check \rho_{u^+}$}; \draw (0.4,0.2) node[right,very
            thick] {$u^-$}; \draw (1.6,2.5) node[below,very thick]
            {$u^+$}; \draw[dashed] (1,1) -- (0.5,2.5); \draw[dashed]
            (1,1) -- (0.3,2.2); \draw[dashed] (1,1) -- (0.1,2); \draw
            (1,1) -- (2.2,1.7); \draw (0.4,0.8) node[below]
            {$\hat \rho_{u^-}$}; \draw (1.,2.2) node[below,]
            {$\hat \rho_{u^+}$};
          \end{tikzpicture}
          \caption{$u^- < u^+ $}
          \label{2WA}
        \end{subfigure}
        \hfill
        \begin{subfigure}[b]{0.45\textwidth}
          \centering
          \begin{tikzpicture}[scale=2]
            \draw[very thick] (0,0) -- (1,1); \draw[very thick] (1,1)
            -- (1.6,2.5); \draw[dashed] (0,1) -- (2.2,1); \draw (0,1)
            node[left] {$t=\bar t$}; \draw (1.2,0.6) node[below]
            {$\check \rho_{u^-}$}; \draw (1.6,1.9) node[below]
            {$\check \rho_{u^+}$}; \draw (0.25,0.2) node[right,very
            thick] {$u^-$}; \draw (1.4,2.5) node[below,very thick]
            {$u^+$}; \draw (1,1) -- (0.5,2.5); \draw[dashed] (1,1) --
            (2.1,1.9); \draw[dashed] (1,1) -- (2.2,1.7); \draw
            (0.1,0.6) node[below] {$\hat \rho_{u^-}$}; \draw (1.,2.2)
            node[below,] {$\hat \rho_{u^+}$};
          \end{tikzpicture}
          \caption{$u^- > u^+ $}
          \label{2WB}
        \end{subfigure}
        \caption{Jump in $u$ at $t=\bar t$ with a $UC_{u^-}$-wave at
          $t=\bar t _-$ and a $UC_{u^+}$-wave at $t=\bar t _+$.}
        \label{2W}
      \end{figure}
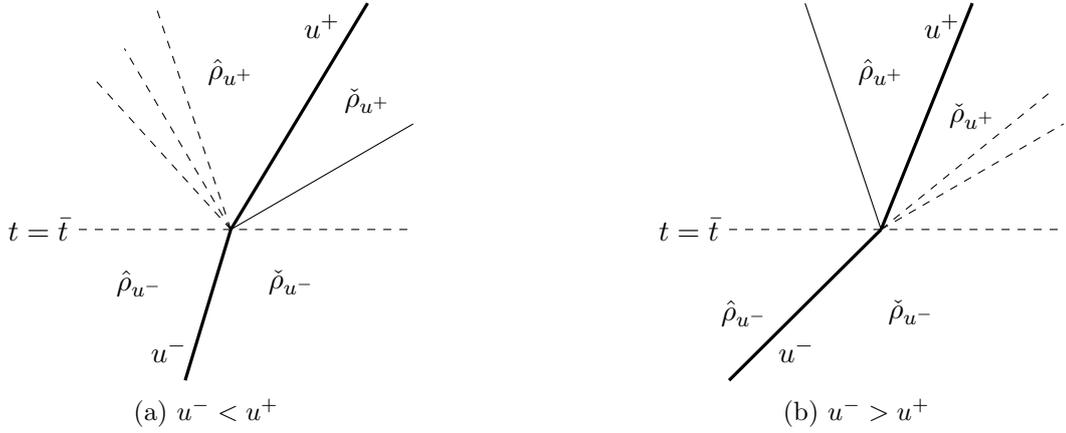
  
    \item $u^+<u^-$ and  $\hat{\rho}_{u^-} \leq \check\rho_{u^+}$
      (see Figure~\ref{NSleft2B}). In this situation we have
      $f'(\hat\rho_{u^-}) \geq u^+$  and 
      $\check \rho_{u^-} \leq \hat \rho_{u-}, \check \rho_{u^+} < \hat
      \rho_{u^+}$.
      If $u^-=V$ (and thus $\check \rho_{u^-} = \hat \rho_{u-}$), we have
      \begin{equation*}
        \Delta \Upsilon (\bar t)\leq -\frac{6}{\beta}\abs{u^+-u^-} \leq
        -\frac{6}{\beta} \delta_u^{\nu}.
      \end{equation*}
      Otherwise, by Taylor formula and~\ref{hyp:F}, there exists
      $\xi\in \ ]\check \rho_{u^-},\hat \rho_{u^-}[$ such that
      \begin{align}
        u^--u^+ 
        &\geq \frac{f(\check\rho_{u^-}) - f(\hat\rho_{u^-})}
          {\check\rho_{u^-}-\hat\rho_{u^-}}-f'(\hat \rho_{u^-})\nonumber \\
        &= -\frac{f''(\xi)}{2}(\hat \rho_{u^-}-\check \rho_{u^-}) \label{lklk2}\\
        &\geq \frac{\beta}{2}  (\hat \rho_{u^-}-\check \rho_{u^-}).\nonumber
      \end{align}
      From \eqref{lk} and \eqref{lklk2},
      \begin{equation*}
        \Delta \Upsilon (\bar t)\leq - \left( \hat \rho_{u^-}-\check
          \rho_{u^-}\right)\leq -\delta_\rho^\nu.
      \end{equation*}
    \end{enumerate}
    
  \item A new undercompressive shock arises at $t=\bar t_+$.
  Again, we distinguish two cases.
  \begin{enumerate}
  \item If $u^-<u^+$ and $\hat{\rho}_{u^+} > \check\rho_{u^-}$, then we have
    $\check \rho_{u^+} < \check \rho_{u^-}< \hat \rho_{u^+} < \hat
    \rho_{u^-}$
    (see Remark~\ref{rem2}). In this case, a $UC_{u^+}$-wave is
    created together with a rarefaction fan
    $(\hat \rho_{u^-},\hat \rho_{u^+})$ and a shock wave
    $(\check \rho_{u^+},\check \rho_{u^-})$ (see Figure~\ref{2WA}).
    Thus, we have
    \begin{align*}
      & \Delta \tv(\rho(\bar t,\cdot))=
        2 \left( \check \rho_{u^-} - \check \rho_{u^+} \right),  \\
      & \Delta \gamma(\bar t)=
        -2\left(\hat \rho_{u^+} - \check \rho_{u^+}\right)
        +2\left(\hat \rho_{u^-} - \check \rho_{u^-}\right),\\
      & \Delta \tv (u,[\bar t,+\infty[)=-\abs{u^+-u^-},
    \end{align*} 
    and we conclude that
    \begin{equation}
      \label{lplp2}
      \Delta \Upsilon (\bar t)=2
      \left(\hat \rho_{u-} - \hat \rho_{u^+}\right)
      -\frac{6}{\beta}\abs{u^+-u^-}.
    \end{equation} 
    Moreover, by Taylor's theorem and~\ref{hyp:F},
    \begin{align*}
      u^+-u^- 
      &\geq  \frac{f(\hat \rho_{u^+})-f(\check \rho_{u^-} )}
        {\hat \rho_{u^+} -\check \rho_{u^-} }
        -\frac{f(\hat \rho_{u^-})-f(\check \rho_{u^-} )}
        {\hat \rho_{u^-} -\check \rho_{u^-} } \\
      &= \left(\frac{f'(\xi)}{\xi-\check \rho_{u^-} }
        -\frac{f(\xi)-f(\check \rho_{u^-} )}{(\xi-\check \rho_{u^-} )^2}\right)
        \left( \hat \rho_{u^+}-\hat \rho_{u^-}\right)
    \end{align*}
    for some $\xi \in \ ] \hat \rho_{u^+}, \hat \rho_{u^-}[$.
    Combining Taylor's theorem with~\ref{hyp:F}, we conclude that there exist
    $\tilde{\xi} \in \ ] \xi, \hat \rho_{u^-}[$ such that
    \begin{align}
      u^+-u^- 
      &\geq - \frac{f''(\tilde{\xi})}{2}( \hat \rho_{u^-}-\hat \rho_{u^+})
      \geq  \frac{\beta}{2} ( \hat \rho_{u^-} - \hat \rho_{u^+}).
      \label{lklk2222}
    \end{align}
    From~\eqref{lplp2} and~\eqref{lklk2222},
    \begin{equation*}
      \Delta \Upsilon (\bar t)\leq - \left(\hat \rho_{u-} - \hat
        \rho_{u^+}\right) \leq -\delta_\rho^\nu.
    \end{equation*}
 
  \item If $u^+<u^-$ and {$\hat\rho_{u^-} > \check\rho_{u^+}$},
then we have
    $\check \rho_{u^-} < \check \rho_{u^+}< \hat \rho_{u^-} < \hat
    \rho_{u^+}$
    (see Remark~\ref{rem2}). In this case, a $UC_{u^+}$-wave is
    created together with a shock wave
    $(\hat \rho_{u^-},\hat \rho_{u^+})$ and a rarefaction fan
    $(\check \rho_{u^+},\check \rho_{u^-})$ (see
    Figure~\ref{2WB}). Thus, we have
    \begin{align*}
      & \Delta \tv(\rho(\bar t,\cdot))=
        2 \left(\hat \rho_{u^+} - \hat \rho_{u^-}\right)  \\
      & \Delta \gamma(\bar t)=
        -2\left(\hat \rho_{u^+} - \check \rho_{u^+}\right)
        +2\left(\hat \rho_{u^-} - \check \rho_{u^-}\right),\\
      & \Delta \tv(u;[\bar t,+\infty[)=-\abs{u^+-u^-},
    \end{align*}
    and we conclude that
    \begin{equation}
      \label{lplp}
      \Delta \Upsilon (\bar t)=2 \left(
        \check\rho_{u^+} -
        \check\rho_{u^-}\right)-\frac{6}{\beta}\abs{u^+-u^-}.
    \end{equation}
    Moreover, by Taylor's theorem and~\ref{hyp:F},
    \begin{align*}
      u^--u^+ 
      &\geq \frac{f(\check \rho_{u^-})-f(\hat \rho_{u^+} )}
        {\check \rho_{u^-} -\hat \rho_{u^+} }
        -\frac{f(\check \rho_{u^+})
        -f(\hat \rho_{u^+} )}{\check \rho_{u^+} -\hat \rho_{u^+} } \\
      &=
        \left(\frac{f'(\xi)}{\xi-\hat \rho_{u^+} }-
        \frac{f(\xi)-f(\hat \rho_{u^+} )}{(\xi-\hat \rho_{u^+} )^2}\right)
        \left( \check \rho_{u^-}-\check \rho_{u^+}\right)
    \end{align*}
    for some $\xi \in \ ] \check \rho_{u^-}, \check \rho_{u^+} [$.
    Combining again Taylor's theorem with~\ref{hyp:F},
    we conclude that there exists
    $\tilde{\xi} \in \ ] \check \rho_{u^-},\xi[$ such that
    \begin{equation} \label{lklk222} u^--u^+ \geq -
      \frac{f''(\tilde{\xi})}{2} \left( \check \rho_{u^+}-\check
        \rho_{u^-}\right) \geq \frac{\beta }{2}\left( \check
        \rho_{u^+}-\check \rho_{u^-}\right).
    \end{equation}
    From \eqref{lplp} and \eqref{lklk222} we obtain
    \begin{equation*}
      \Delta \Upsilon (\bar t) \leq - \left( \check\rho_{u^+} -
        \check\rho_{u^-}\right)\leq -\delta_\rho^\nu.
    \end{equation*}
  \end{enumerate}
\end{enumerate}
The proof is so finished.
\end{proof}
This concludes the proof of Proposition~\ref{prop:upsilon}.
\end{proof}

Proposition~\ref{prop:upsilon} ensures that the number of wave fronts in $\rho_\nu$
is finite and the wave-front tracking procedure can be prolonged for every positive time. 
Moreover, the total variation of $\rho_\nu$ is uniformly bounded in time:
\begin{corollary} \label{cor:TVbound}
  Let $\left\{\rho_\nu,y_\nu,u_\nu\right\}_{\nu\in\N}$ be the sequence
  of approximate solutions constructed in Section~\ref{sse:wft}.
  Then for any $\nu\in\N$, $t>0$, it holds
  \begin{equation*} 
    \tv\left(\rho_\nu(t,\cdot)\right) \leq \Upsilon(t)\leq
    \Upsilon(0) \leq \tv(\rho_0)+2R+\tv(u).
  \end{equation*}
\end{corollary}

\subsection{Existence of a solution}
\label{sse:existence-solution}
The following is the main result of the paper.
\begin{theorem} \label{final21}
  Let the initial conditions $\rho_0 \in \LL1 
  \left(\R; [0,R]\right)$ with finite total variation,
  $y_0 \in \R$, and the open-loop control
  $u \in \LL1\left(\R^+; [0, V] \right)$ with finite total variation.
  There exists a solution $\left(\rho, y\right)$
  to~\eqref{eq:control-model} in the sense of Definition~\ref{def:sol-CP}.
\end{theorem}

First, we prove that a limit of the sequence of approximate solutions
$\left\{\rho_\nu,y_\nu,u_\nu\right\}_{\nu\in\N}$ to~\eqref{eq:control-model}
constructed in Section~\ref{sse:wft} exists. 
\begin{lemma} \label{convergence-gen}
  Let $\left\{\rho_\nu,y_\nu,u_\nu\right\}_{\nu\in\N}$ be the sequence of
  approximate solutions 
  to~\eqref{eq:control-model} constructed in Section~\ref{sse:wft}.
  Then, up to a subsequence, we have 
  \begin{subequations}
    \label{convergence}
    \begin{align}
      & \rho_\nu  \to \rho,
      & \textrm{in } \, \Ll1 (\R^+ \times \R; [0,R]),
        \label{convergence-rho}\\[5pt]
      & y_\nu \to  y,
      & \textrm{in } \, \Ll{\infty}(\R^+; \R),
        \label{convergence-y}\\[5pt]
      & \dot y_\nu \to \dot y,
      & \textrm{in } \, \Ll{1}(\R^+; \R), \label{convergence-doty}
    \end{align}
  \end{subequations}
  for some $\rho \in \C0\left(\R^+;
    \LL1 (\R; [0,R])\right)$ with $\tv\left(\rho(t)\right) < + \infty$
  for a.e. $t \in \R^+$
  and $y \in \Wl{1,1}(\R^+;\R)$, which is a Lipschitz continuous function
  with Lipschitz constant $V$. 
\end{lemma}

\begin{proof}
  Fix $T > 0$. 
  By Corollary~\ref{cor:TVbound}, we know that $\tv(\rho_\nu(t,\cdot))$
  is uniformly bounded for a.e. $t \in [0, T]$.  
  This, together with the finite wave speed propagation, implies that
  \begin{equation*}
    \int_\R \abs{\rho_\nu(t,x) - \rho_\nu(s,x) } dx
    \leq L \abs{t-s} \qquad \hbox{for all}~t,s \in [0, T],
  \end{equation*}
  for some $L$ depending on $f$ and on the total variation bound,
  but not on $\nu$.
  Helly's Theorem,
  see~\cite[Theorem~2.4]{Bressan_book_hyp_2000},
  implies the existence of
  $\rho \in \C{0}\left([0, T]; \LL1 
    \left(\R; [0,R]\right)\right)$
  such that $\tv\left(\rho(t)\right) < + \infty$ for a.e. $t \in [0, T]$
  and a subsequence of $\{ \rho_\nu\}_\nu$, which for simplicity we denote
  again by $\{\rho_\nu\}_\nu$. This implies that~(\ref{convergence-rho})
  holds.

  By construction (see point~\ref{point:7}
  of Definition~\ref{def:epswf}), we deduce that
  \begin{equation}
    \label{eq:bound_dot_y}
    0 \le \dot{y}_\nu(t) \le V
  \end{equation}
  for a.e. $t \in [0, T]$ and $\nu \in \N \setminus\left\{0\right\}$.
  Hence Ascoli Theorem~\cite[Theorem~7.25]{MR0385023}
  implies that, there exists a function $y \in \C{0}\left([0,T]; \R\right)$
  and a subsequence of $\{y_\nu\}_\nu$, which for simplicity we denote
  again by $\{y_\nu\}_\nu$, such that $y_\nu$ converges to
  $y$ uniformly in $\C{0}\left([0,T]; \R\right)$.
  By the arbitrariness of $T$, the function $y$ can
  be defined on $\R^+$ and~(\ref{convergence-y}) holds. Moreover
  $y$ is a Lipschitz continuous function with $V$ as a Lipschitz constant.

  To prove~(\ref{convergence-doty}), we aim to estimate the total variation
  of $\dot y_\nu$ on $[0, T]$.
  Since~(\ref{eq:bound_dot_y}) holds, then it is sufficient
  to estimate the positive variation of $\dot y_\nu$, denoted with the
  symbol $\tvp$.
  Observe that $\dot y_\nu$ can jump only at interactions
  with waves coming from the right 
  (see the interactions' estimates in~\cite[Section 4.2]{MR3264413})
  or at jumps in the control $u_\nu$.
  In particular, $\dot y_\nu$ is non-decreasing at interactions
  with rarefaction fronts, which can be originated 
  at $t=0$, or at upward jumps in $u_\nu$ (see Figure~\ref{NSleft2A}).
  We have the following two possibilities.
  
  \begin{enumerate}[label=\Roman*.]
  \item \textbf{A fan of rarefaction fronts interacts with the AV
      at most once.}
    If $y_\nu$ interacts over the time interval $[t_1, t_2]$
    with rarefaction fronts, all originated from the point $(0,x_0)$,
    and $u_\nu$ is constant in $[t_1, t_2]$, then there exists
    a constant $C>0$ such that
    \begin{equation}
      \label{eq:final-0}
      \tvp (\dot y_\nu; [t_1,t_2]) \le
      \abs{v(\rho_\nu(t_1,y_\nu(t_1)_-)) - v(\rho_\nu(t_2,y_\nu(t_2)_+))}
      \le C \abs{\rho_0({x_0}_-) - \rho_0({x_0}_+)}.
    \end{equation}

    If $y_\nu$ interacts over $[t_1,t_2]$ with rarefaction fronts,
    all originated from $(\bar t, y_\nu(\bar t))$ with $\bar t>0$, and $u_\nu$
    is constant over $ [t_1,t_2]$ (such rarefaction wave created at $\bar t>0$
    is described in Lemma~\ref{Y2}; see also Figure~\ref{NSleft2}
    and Figure~\ref{2W}), then either a $\mathcal{F}_{u_{\nu}(\bar t_+)}$-wave
    arises
    (see Figure \ref{NSleft2}) or a $UC_{u_\nu(\bar t_+)}$-wave arises
    (see Figure \ref{2W}).
    In the former case, the rarefaction wave connects
    $\hat \rho_{u_\nu(\bar t_-)}$ to
    $\check \rho_{u_\nu(\bar t_-)}$ and so, using~\eqref{lklk} and~\eqref{lklk2},
    there exists a positive constant $C > 0$ such that
    \begin{equation}
      \label{final1}
      \tvp(\dot y_\nu; [t_1,t_2]) \le
      \abs{v\left(\check \rho_{u_\nu(\bar t_-)}\right)
        - v\left(\hat \rho_{u_\nu(\bar t_-)}\right)}
      \leq C \abs{u_\nu(\bar t_+)-u_\nu(\bar t_-)}.
    \end{equation}
    
    In the latter case, when a $UC_{u_\nu(\bar t_+)}$-wave arises
    (see Figure \ref{2W}),
    the rarefaction wave connects either $\hat \rho_{u_\nu(\bar t_-)}$ to
    $\hat \rho_{u_\nu(\bar t_+)}$ or $\check \rho_{u_\nu(\bar t_-)}$ to
    $\check \rho_{u_\nu(\bar t_+)}$.
    Using~\eqref{lklk2222} and~\eqref{lklk222}, we deduce that
    there exists a constant $C>0$ such that
    \begin{equation} \label{final2}
      \tvp (\dot y_\nu; [t_1,t_2]) \leq
      \abs{v(\rho_\nu(t_1,y_\nu(t_1))) - v(\rho_\nu(t_2,y_\nu(t_2)))}
      \leq C \abs{u_\nu(\bar t_+) - u_\nu(\bar t_-)}.
    \end{equation}
  \item \textbf{The AV interacts with a fan of rarefaction shocks
      $(\rho_l,\rho_r)$ modifying $n_\nu$ times its speed.}
    \begin{figure}[!ht]
      \centering
      \includegraphics[width=0.45\linewidth]{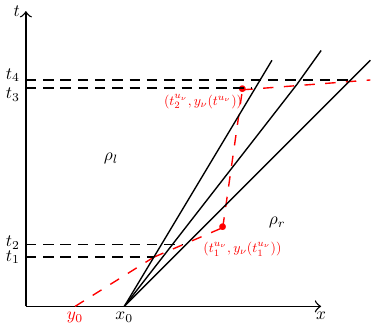}
      \caption{The AV interacts with a fan of rarefaction shocks
        $(\rho_l, \rho_r)$ modifying its speed over the time intervals
        $[t_1,t_2]$ and $[t_3,t_4]$. In this situation $n_\nu = 2$ and,
        for every $i \in \{1,2\}$, $u_\nu({t_i^{n_\nu}}_-)\neq
        u_\nu({t_i^{n_\nu}}_+)$.
      }\label{ex1}
    \end{figure}
    Define, for every $i \in \{ 1,\cdots,n_\nu\}$, the times $t_{2i-1}$ and
    $t_{2i}$, respectively the times at which the AV enters from the left
    and exits from the right the rarefaction fan.
    Moreover, for every $i \in \left\{1, \cdots, n_\nu\right\}$,
    define by $t^{u_\nu}_{2i-1}$ and $t^{u_\nu}_{2i}$
    the point of discontinuity for $u_\nu$ in such a way
    \begin{equation*}
      t_0 < t_1 < t_2 < t_1^{u_\nu} < t_2^{u_\nu} < \cdots
      < t_{2n_\nu - 3} < t_{2n_\nu - 2}< t^{u_\nu}_{2n_\nu -3}
      < t^{u_\nu}_{2n_\nu - 2} < t_{2 n_\nu - 1} < t_{2 n_\nu},
    \end{equation*}
    where $t_0$ is the time at which the rarefaction fan is originated.
    For simplicity we denote with $\rho_l$ and $\rho_r$ respectively the
    left and the right states of the rarefaction fan;
    see Figure~\ref{ex1}. Note that
    $v\left(\rho_l\right) < v\left(\rho_r\right)$.

    Since, $\dot y_\nu$ can increase only at interactions with waves
    coming from the right or at jumps in $u_\nu$, we have 
    \begin{align*}
      \tvp(\dot y_\nu;[t_1,t_{2n_\nu}])
      & \le \sum_{i=1}^{n_\nu} \abs{v\left(\rho_l\right) - v\left(\rho_r\right)}
        + \sum_{i=1}^{2 n_\nu-2} \vert
        u_\nu({t_i^{u_\nu}}_-) - u_\nu({t_i^{u_\nu}}_+)\vert
      \\
      & =  n_\nu \left[v(\rho_r) - v(\rho_l)\right]
        + \sum_{i=1}^{2 n_\nu-2} \vert
        u_\nu({t_i^{u_\nu}}_-) - u_\nu({t_i^{u_\nu}}_+)\vert.
    \end{align*}
    The speed of AV is modified in the time interval
    $\left[t_{2i-1}, t_{2i}\right]$ for every $i \in \{1, \cdots, n_\nu\}$;
    in particular, we deduce that
    $v(\rho_l) \leq v(\rho_r) \leq u_\nu({t_{2i-1}^{u_\nu}}_-)$
    for every $i \in \{1, \cdots, n_\nu-1\}$.
    Moreover, since the speed of AV is constant over
    the time interval $\left]t_{2i-1}^{u_\nu},t^{u_\nu}_{2i}\right[$
    for every $i \in \left\{1, \cdots, n_\nu - 1\right\}$, we have
    $u_\nu({t_{2i-1}^{u_{\nu}}}_+) \le v(\rho_l) \le v(\rho_r)$. Therefore
    \begin{equation*}
      (n_\nu - 1) \left[v(\rho_r) - v(\rho_l)\right]
      \le \sum_{i=1}^{n_\nu-1}
      \abs{u_\nu({t_{2i-1}^{u_\nu}}_-) - u_\nu({t_{2i-1}^{u_\nu}}_+)},
    \end{equation*}
    and so
    \begin{equation}
      \label{final3}
      \tvp (\dot y_\nu;[t_1,t_{2n_\nu}]) \leq \left[v(\rho_r) - v(\rho_l)\right]
      +2\sum_{i=1}^{u_\nu - 1}
      \abs{u_\nu({t_{2i-1}^{u_\nu}}_-) - u_\nu({t^{n_\nu}_{2i-1}}_+)}.
    \end{equation}
  \end{enumerate}
 We have  
\begin{align}
  \tv \left(\dot y_\nu; [0,T]\right)
  & \leq 2 \tvp
    \left( \dot y_\nu; [0,T]\right) + \norm{\dot y_\nu}_{\infty},
    \nonumber\\
  & \leq 2\sum_{i=0}^{N_\nu}\tvp(\dot y_\nu;]t_i^{u_\nu},t_{i+1}^{u_\nu}[)
    + 2 \sum_{i=1}^{N_\nu}
    \vert u_\nu({t_i^{u_\nu}}_+)-u_\nu({t_i^{u_\nu}}_-)\vert+
    \norm{\dot y_\nu}_{\infty}.\label{ffoh1}
 \end{align}
Above,   $(t_i^{u_\nu})_{i=1,\cdots,N_\nu}$  are the $N_\nu$ discontinuous points of $u_\nu$ such that, for every $i=1,\cdots,N_\nu$, $t_i^{u_\nu}< T$. Moreover,  by convention $t_{0}^{u_\nu}=0$ $t_{N_\nu+1}^{u_\nu}=T$. Combining~(\ref{eq:final-0}), \eqref{final1}, \eqref{final2},
  and~\eqref{final3}, we deduce 
  \begin{equation}\label{ffoh}\sum_{i=0}^{N_\nu}\tvp(\dot y_\nu;]t_i^{u_\nu},t_{i+1}^{u_\nu}[)\leq C(TV(\rho_0)+TV(u))+2TV(u).\end{equation}
 From \eqref{ffoh1} and \eqref{ffoh}, there exists a constant 
  $C>0$ such that 
  \begin{align*}
  \tv \left(\dot y_\nu; [0,T]\right)  \leq 2C\tv(\rho_0)+(2C+6)\tv(u)+V,
 \end{align*}
  proving~(\ref{convergence-doty}), concluding the proof.
\end{proof}

\begin{proofof}{Theorem \ref{final21}}
  Consider a sequence of approximate solutions
  $\left\{\rho_\nu, y_\nu, u_\nu\right\}_{\nu\in\N}$ to~\eqref{eq:control-model}
  constructed in Section~\ref{sse:wft}.
  By Lemma~\ref{convergence-gen}, there exist 
  $\rho \in \C0\left(\R^+;
    \LL1(\R; [0,R])\right)$ such that $\tv\left(\rho(t)\right) < + \infty$
  for a.e. $t$
  and $y \in \Wl{1,1}(\R^+;\R)$ such that, up to a subsequence,
  (\ref{convergence-rho}), (\ref{convergence-y}), and~(\ref{convergence-doty})
  hold.
  We prove that the couple $\left( \rho,  y\right)$ provides
  a solution to~\eqref{eq:control-model} according to
  Definition~\ref{def:sol-CP}.    
  Clearly the points~\ref{def:sol-CP:1} and~\ref{def:sol-CP:2}
  of Definition~\ref{def:sol-CP} hold.
  
  Since $\rho_\nu$ is a weak entropy solution of~\eqref{eq:LWR}--\eqref{eq:constraint}--\eqref{eq:initial} in the sense of Definition~\ref{def:sol-CP}, points~\ref{def:sol-CP:3} and~\ref{def:sol-CP:5}, then, for every $\kappa\in\R$ and for all $\varphi\in\Cc1 (\R^2;\R^+)$, it holds
\begin{align}
    \int_{\R^+}\int_\R &\left(
    |\rho_\nu-\kappa| \pt \varphi +\sgn (\rho_\nu-\kappa)(f(\rho_\nu)-f(\kappa)) \px\varphi
    \right) dx\, dt 
    +\int_\R |\rho_{0,\nu}-\kappa| \varphi(0,x) \, dx \nonumber \\
    &{+}\,  2\int_{\R^+} \left(
    f(\kappa)  - \dot y_\nu (t) \kappa - \min\{f(\kappa) - \dot y_\nu (t) \kappa, F_\alpha(\dot y_\nu(t))\} 
    \right) \varphi(t,y(t)) \, dt \geq 0\,. \label{eq:ECapprox}
\end{align}
  Using~\eqref{convergence-rho} and~\eqref{convergence-doty} and passing to the limit in~\eqref{eq:ECapprox}
  as $\nu \to +\infty$, we conclude that $\rho$ satisfies~\eqref{eq:EC}, hence point~\ref{def:sol-CP:3}
  of Definition~\ref{def:sol-CP} holds.

  We deal now with the point~\ref{def:sol-CP:5}
  of Definition~\ref{def:sol-CP}. Fix $T > 0$ and consider the sets
  \begin{align*}
    D_l := \left\{(t,x)\in [0,T]\times \R :\, x < y(t)\right\},
    \quad
    &
      D_r := \left\{(t,x)\in [0,T]\times \R :\, x > y(t)\right\},
    \\
    D_l^\nu := \left\{(t,x)\in [0,T]\times \R :\, x < y_\nu(t)\right\},
    \quad
    &
      D_r^\nu := \left\{(t,x)\in [0,T]\times \R :\, x > y_\nu(t)\right\}.
  \end{align*}
Fix $\psi \in \Cc1\left(\left]0, T\right[ \times \R; \R^+\right)$. 
By~\cite[Theorem 2.2]{ChenFrid1999},
since $\rho_\nu$ and $\rho$ are weak solutions of \eqref{eq:LWR-model}, 
we deduce that      
\begin{equation}
  \label{test}
  \int_{D_l^\nu} 
  \left(\rho_\nu \pt \psi + f(\rho_\nu) \px \psi\right)
  dt\,dx
  = \int_0^T \left[f(\rho_\nu(t,y_\nu(t)_-)) -
    \rho_\nu(t,y_\nu(t)_-)\dot y_\nu(t)\right]\psi(t,y_\nu(t)) dt
\end{equation}
and
\begin{equation}
  \label{test4}
  \int_{D_l} 
  \left(\rho \pt \psi + f(\rho) \px \psi\right)
  dt\,dx
  = \int_0^T \left[f(\rho(t,y(t)_-)) -
    \rho(t,y(t)_-)\dot y(t)\right]\psi(t,y(t)) dt
\end{equation}
hold.  The construction of $(\rho_\nu, y_\nu)$, \eqref{test} and the fact that
  $\psi \ge 0$ imply
  \begin{equation}
    \label{test0}
    \int_{D_l^\nu} 
      \left(\rho_\nu \pt \psi + f(\rho_\nu) \px \psi\right)
      dt\,dx
    \leq \int_0^T F_{\alpha}(\dot y_\nu (t)) \psi(t,y_\nu(t)) dt.
  \end{equation}
  Lemma~\ref{convergence-gen} and the Dominated Convergence Theorem
  imply
  \begin{equation}
    \label{test1}
  \lim_{\nu\to \infty} \int_{D_l^\nu} \left(\rho_\nu \pt \psi + f(\rho_\nu) \px \psi\right) dt\, dx = \int_{D_l} \left(\rho \pt \psi + f(\rho) \px \psi\right) dt\, dx,
 \end{equation} 
  and
  \begin{equation}
    \label{jvi}
    \lim_{\nu\to \infty} \int_0^T F_{\alpha}(\dot y_\nu(t)) \psi(t,y_\nu(t)) dt
    = \int_0^T F_{\alpha}(\dot y(t)) \psi(t,y(t)) dt.
  \end{equation}
  Therefore, using~\eqref{test4}, \eqref{test0},
  \eqref{test1}, \eqref{jvi}, we get
  \begin{equation*}
    \int_0^T \left[f(\rho(t,y(t)_-))-\rho(t,y(t)_-)\dot y(t)\right]
    \psi(t,y(t)) dt
    \leq  \int_0^T F_{\alpha}(\dot y(t)) \psi(t,y(t)) dt.
  \end{equation*}
  The same holds for the right traces.
  By the arbitrariness of $\psi$, we deduce that 
  \begin{equation*}
    f(\rho(t,y(t)_\pm)) - \rho(t,y(t)_\pm)\dot y(t)
    \leq F_{\alpha}(\dot y(t))
  \end{equation*}
  for a.e $t\in \, ]0,T]$. Thus the couple $\left(\rho, y\right)$
  satisfies point~\ref{def:sol-CP:5} of Definition~\ref{def:sol-CP}.
   
   \medskip
   
  It remains to prove that the couple
  $\left(\rho, y\right)$ satisfies point~\ref{def:sol-CP:4} of
  Definition~\ref{def:sol-CP}. From Lemma \ref{convergence-gen}, \eqref{eq:uapprox} and the construction of $(\rho_\nu,y_\nu,u_\nu)$, there exists a null set $\mathcal{N}$ such that, for every $\bar t\in \R^*_+\setminus \mathcal{N}$,
  \begin{itemize}
  \item
    $\displaystyle \lim_{\nu\to \infty} \dot y_\nu(\bar t\,) = \dot
    y(\bar t\,)$, 
     
  \item
    $\dot y_\nu(\bar t\,) = \min\left\{u_\nu(\bar t\,), v
      \left(\rho_\nu(\bar t,y_\nu(\bar t\,) _+ )\right)\right\}$,
     
  \item $y$ is continuously differentiable at $\bar t$,

  \item
    $\displaystyle \lim_{\nu\to \infty} \rho_\nu(\bar t, x) = \rho(
    \bar t,x)$ for a.e. $x \in \R$,
     
  \item
    $\displaystyle \lim_{\nu\to \infty} u_\nu(\bar t\,) = u(\bar
    t\,)$,
     
  \item $u(\bar t_-) = u(\bar t_+)=:\bar u$.
  \end{itemize}
We have to prove that, if $\bar t \in \R^+ \setminus \mathcal{N}$, then
  \begin{equation}
    \label{eqfinal}
    \lim_{\nu\to \infty} \min \left\{u_\nu(\bar t\,), 
      v \left(\rho_\nu( \bar t,y_\nu( \bar t)_+)\right) \right\}
    = \min \left\{u(\bar t\,), 
      v \left(\rho( \bar t, y( \bar t)_+)\right) \right\}.
  \end{equation}
  To this aim, it is sufficient to prove that 
  $\lim_{\nu\to \infty}v \left(\rho_\nu( \bar t,y_\nu( \bar t)_+)\right)
  = v \left(\rho( \bar t, y( \bar t)_+)\right)$.
  The proof follows closely the one given in \cite[Section 3.3]{LiardPiccoli2} for a constant control speed $u$.
  
  Define 
  $\rho_{\pm} = \lim_{x\to y(\bar t)\pm}\rho(\bar t,x)$, which
  exist since $\tv(\rho(\bar t, \cdot); \R)$ is finite.  
  Various cases can occur. 
  
  \begin{enumerate}
  \item Case: $\rho_-,\rho_+ \in [\rho^*_{\bar u},R]$. 
    From Lemma~\ref{sup}, see also the entropy condition~\eqref{eq:EC},
    the only possible case is $\rho^*_{\bar u}\leq \rho_-\leq \rho_+$.
      
\begin{enumerate}

\item If $\rho_+=\rho_-$, then using Lemma~\ref{suprhostar} and
  Lemma~\ref{suprhostar2}, we have that for every
  $\epsilon>0$ there exists $\bar\nu\in\N$ such that 
  \begin{equation}\label{hahabien}
    v(\min\{\rho_++2\epsilon,R\}) \!\leq\! \min\{u_\nu(\bar t),
    v(\rho_\nu(\bar t,y_\nu(\bar t)_+))\}:=\dot y_\nu(\bar t) \leq
    \min\{u_\nu(\bar t), v(\rho_+-2\epsilon)\}
  \end{equation}
  for every $\nu\geq\bar\nu$.
Since $\bar t \in \R_+^*\setminus \mathcal{N}$, by passing to the limit in \eqref{hahabien} as $\nu\to \infty$, we deduce~\eqref{eqfinal}, for the arbitrariness of $\epsilon$.

\item  If $\rho_+\neq \rho_-$ and $y(\bar t)\leq y_\nu(\bar t)$ up to a subsequence; from Lemma \ref{suprhostar} we have that for every $\epsilon>0$ there exists $\bar\nu\in\N$ such that 
  \begin{equation}\label{hahabien2}
    v(\min\{\rho_++\epsilon,R\}) \leq
    \min\{u_\nu(\bar t),v(\rho_\nu(t,y_\nu(\bar t)_+))\}
    :=\dot y_\nu(\bar t) \leq v(\rho_+-\epsilon)
\end{equation}
for every $\nu\geq\bar\nu$.
Since $\bar t \in \R_+^*\setminus \mathcal{N}$, the equality \eqref{eqfinal} holds by passing to the limit in \eqref{hahabien2} as $\nu\to \infty$. 
\item \label{item:tricky}
If $\rho_+\neq \rho_-$ and $y_\nu(\bar t) < y(\bar t)$ up to a subsequence; in this case, from Lemma \ref{suprhostar} and Lemma \ref{suprhostar2}, 
we have that for every $\epsilon>0$ there exists $\bar\nu\in\N$ such that 
$\rho_\nu(\bar t,y_\nu(\bar t)_+) \in \,]\rho_--2\epsilon,\rho_++2\epsilon[$
for every $\nu\geq\bar\nu$. 

For a.e $t>\bar t$, 
\begin{equation}\label{eq:intn}y_\nu(t)-y_\nu(\bar t )=\int_{\bar t}^t \dot y_\nu(s) ds
\end{equation}
and, by Lemma \ref{convergence-gen}, 
\begin{equation*}
  \lim_{\nu\to \infty} y_\nu(t)=y(t).
\end{equation*}
By Lemma~\ref{oooo}, there exists $c>0$  and a sequence
$t_\nu\searrow\bar t$ such that
$\rho_\nu (s, y_\nu(s)_+)\in \,]\rho_+-\epsilon,\rho_++\epsilon[$ for every $s\in [t_\nu,\bar t+c]$, hence $\dot y_\nu(s)\in \,]v(\rho_++\epsilon),v(\rho_+-\epsilon)[$ for every $s\in [t_\nu,\bar t+c]$.
By passing to the limit in \eqref{eq:intn}, we have  for a.e $t\in \, ]\bar t,\bar t+c]$
\begin{equation*}
    \frac{y(t)-y(\bar t )}{t-\bar t}\in [ v(\rho_++\epsilon), v(\rho_+-\epsilon)].
\end{equation*}
Using that $y$ is continuously differentiable at time $\bar t$ and the
arbitrariness of $\epsilon$, we conclude that~\eqref{eqfinal} holds.
\end{enumerate}

\item Case: $\rho_-,\rho_+\in [0,\rho^*_{\bar u}]$.
  Since $\bar t\in \R_+^*\setminus \mathcal{N}$, 
  \begin{equation*}
    \dot y(\bar t)=\lim_{\nu \to \infty} y_\nu(\bar t)=\lim_{\nu\to \infty}
    \min\{u_{\nu}(\bar t),v(\rho_\nu(\bar t,y_\nu(\bar t)_+)\}.
  \end{equation*}
  From Lemma~\ref{rhon1}, for every $\epsilon>0$ there exists $\bar\nu\in\N$
  such that
  \begin{equation*}
    v(\rho^*_{\bar u}+\epsilon)\leq \min\{u_{\nu}(\bar t ),
    v(\rho_\nu(\bar t,y_\nu(\bar t)_+)\}\leq u_{\nu}(\bar t)
  \end{equation*}
  for $\nu\geq\bar\nu$. Since $\rho_+\in [0,\rho^*_{\bar u}]$,
  we know that 
  \begin{equation*}
    \dot y(\bar t) = u(\bar t) = \min\{u(\bar t),v(\rho(\bar t,y(\bar t)_+))\}.
  \end{equation*}

\item Case:  $\rho_+<\rho^*_{\bar u}<\rho_-$. This case cannot
  occur by Lemma~\ref{ine}. 

\item Case: $\rho_-<\rho^*_{\bar u}<\rho_+$.
\begin{enumerate}
   
\item If $y(\bar t)\leq y_\nu(\bar t)$ up to a subsequence, by Lemma \ref{suprhostar} for every $\epsilon>0$ there exists $\bar\nu$ such that we have
  \begin{equation}\label{hahabienbien}
    v(\min\{\rho_++\epsilon,R\}) \leq
    \min(u_{\nu}(\bar t),v(\rho_\nu(t,y_\nu(\bar t)_+)))
    :=\dot y_\nu(\bar t) \leq v(\rho_+-\epsilon)
  \end{equation}
  for $\nu\geq\bar\nu$.
  Since $\bar t \in \R_+^*\setminus \mathcal{N}$,
  the equality \eqref{eqfinal} holds by passing to the limit in
  \eqref{hahabienbien} as $\nu \to \infty$. 

\item If $y_\nu(\bar t) < y(\bar t)$ up to a subsequence, using Lemma~\ref{ooooa} and reasoning as in item~\ref{item:tricky} we get the conclusion.
\end{enumerate}
\end{enumerate}
This concludes the proof.
\end{proofof} 

\section*{Acknowledgments}
The authors thank Boris Andreianov for useful remarks on entropy conditions.
MG was partially supported by the GNAMPA 2017 project
\textsl{Conservation Laws: from Theory to Technology} and by the PRIN 2015
project \textsl{Hyperbolic Systems of Conservation Laws
  and Fluid Dynamics: Analysis and Applications}.
 PG, TL and BP were partially supported by the Inria Associated Team
 {\it ORESTE - Optimal REroute Strategies for Traffic managEment} (2012-2017). 
 BP acknowledges the support of the National Science Foundation under Grants No. CNS-1837481.
 The work of TL has also been partially funded by the European Research Council (ERC) under the European Union's Horizon 2020 research and innovation program (grant agreement No. 694126-DyCon). 
 
%
%
 
\begin{appendices}
 
  \section{Technical lemmas}
  \label{sec:app:tech-lemmas}
  Here we state some technical lemmas used in the proof of
  Theorem~\ref{final21}. In the following, we denote by $\rho$ and $y$
  respectively the limit functions of wave front tracking approximate
  solutions $\rho_\nu$ and $y_\nu$; see Lemma~\ref{convergence-gen}.
  Moreover, if $\bar t > 0$, then we define
  $\rho_- := \rho\left(\bar t, y(\bar t)_-\right)$ and
  $\rho_+ := \rho\left(\bar t, y(\bar t)_+\right)$.
  
  \begin{lemma}
    \label{suprhostar}
    \textbf{\textup{\cite[Lemma 4]{LiardPiccoli2}}}.
    Let $\bar t \in \R^+ \setminus \mathcal{N}$ and $\epsilon>0$.
    Assume that $\rho_-,\rho_+\in [0,R]$.
    There exist $0 < \tilde \delta < \delta$ such that 
    \begin{equation*}
      \rho(\bar t,x) \in
      \left\{
        \begin{array}{ll}
          \left] \max\{\rho_--{\epsilon}/{2},0\},
          \min\{\rho_-+{\epsilon}/{2},R\}
          \right[,
          &
            x \in \, ]y(\bar t)-\delta, y(\bar t)[,
            \vspace{.2cm}
          \\
          \left] \max\{\rho_+-{\epsilon}/{2},0\},
          \min\{\rho_++{\epsilon}/{2},R\}\right[,
          &
            x\in\, ]y(\bar t),y(\bar t)+\delta[,
        \end{array}
      \right.
    \end{equation*}
    and, for $\nu\in \N^*$ sufficiently large, 
    \begin{equation*}
      \rho_\nu(\bar t,x) \!\in\!
      \left\{
        \!\!
        \begin{array}{ll} 
          \left] \max\{\rho_--\epsilon,0\},
          \min\{\rho_-+\epsilon,R\}\right[,
          &
            x\in \,]\min\{y(\bar t),y_\nu(\bar t)\}
            - \tilde{\delta},\min\{y(\bar t),y_\nu(\bar t)\}[,
          \vspace{.2cm}\\
          \left] \max\{\rho_+-\epsilon,0\},
          \min\{\rho_++\epsilon,R\}\right[,
          &
            x\in \,]\max\{y(\bar t),y_\nu(\bar t)\},
            \max\{y(\bar t),y_\nu(\bar t)\}+\tilde \delta[;
        \end{array}
      \right.
    \end{equation*}
    see Figure~\ref{2i}.
  \end{lemma}
  \begin{figure}[h!]
    \centering
    \includegraphics[scale=1.5]{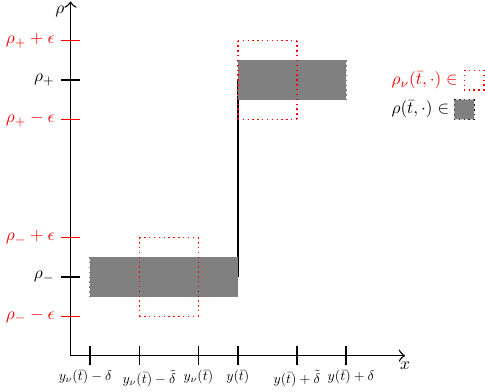}
    \caption{Illustration of Lemma~\ref{suprhostar} in the case
      $\rho_- < \rho_+$ and $y_\nu(\bar t) < y(\bar t)$.
      The approximate density $\rho_\nu(\bar t,\cdot)$
      in the space interval
      $]y_\nu(\bar t)-\tilde \delta,y_\nu(\bar t)[ \, \cup
      \, ]y(\bar t),y(\bar t)+\tilde \delta[$
      belongs to the area surrounded by the red dotted lines, while
      $\rho(\bar t,\cdot)$ in the interval
      $]y(\bar t)- \delta,y(\bar t)+ \delta[$ belongs to the grey shaded zone.}
    \label{2i}
  \end{figure}
  
  \begin{lemma}
    \label{sup}
    Let $\bar t \in \R^+ \setminus \mathcal{N}$.
    If $\rho_-,\rho_+\in [\rho^*_{\bar u},R]$, then $\rho_- \leq \rho_+$.
  \end{lemma}
  The proof is identical to that of~\cite[Lemma 5]{LiardPiccoli2} remarking
  that, for every $\epsilon > 0$, there exists $\bar \nu \in \N$ such that
  $\rho^*_{u_\nu(\bar t)} < \rho^*_{\bar u} + \epsilon \leq \rho_+ + \epsilon$
  for every $\nu \ge \bar \nu$.

  \begin{lemma}
    \label{suprhostar2}
    Let $\bar t \in \R^+ \setminus \mathcal{N}$ and $\epsilon>0$.
    Assume that $\rho_{\bar u}^*\leq \rho_- \leq \rho_+$.
    Then there exists $\bar \nu \in \N$ such that
    \begin{equation*}
      \rho_\nu(\bar t,x) \in 
      \,]\rho_--2\epsilon,\min(\rho_++2\epsilon,R)[,
    \end{equation*}
    for every $\nu \ge \bar \nu$ and for every
    $x\in \,]\min\{y(\bar t),y_\nu(\bar t)\},
    \max\{y(\bar t),y_\nu(\bar t)\}[$; see Figure~\ref{2ii}.
  \end{lemma}
  The proof is identical to that of~\cite[Lemma 6]{LiardPiccoli2} remarking
  that, for every $\epsilon > 0$, there exists $\bar \nu \in \N$ such that
  $\hat \rho_{u_\nu(\bar t)} < \hat \rho_{\bar u} + \epsilon
  \leq \rho^*_{\bar u} - 2 \epsilon$
  for every $\nu \ge \bar \nu$.
  \begin{figure}[h!]
    \centering
    \includegraphics[scale=1.5]{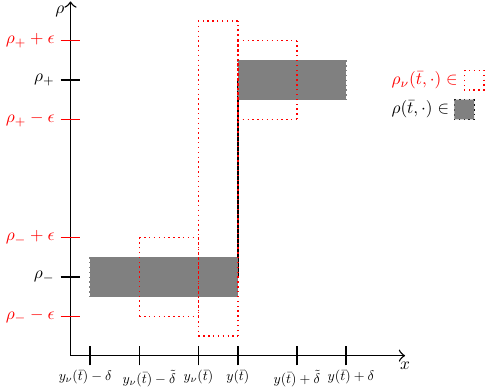}
    \caption{Illustration of Lemma~\ref{suprhostar2} in the case
      $\rho^*_{\bar u} \leq \rho_- < \rho_+$ and
      $y_\nu(\bar t)<y(\bar t)$.
      The approximate density $\rho_\nu(\bar t,\cdot)$
      in the space interval
      $]y_\nu(\bar t)-\tilde \delta,y(\bar t)+\tilde \delta[$
      belongs to the area surrounded by the red dotted lines,
      while $\rho(\bar t,\cdot)$ in the interval
      $]y(\bar t)- \delta,y(\bar t)+ \delta[$ belongs to the grey shaded zone.}
    \label{2ii}
  \end{figure}

  \begin{lemma}
    \label{oooo} 
    \textbf{\textup{\cite[Lemma 7]{LiardPiccoli2}}}.
    Let $\bar t \in \R^+ \setminus \mathcal{N}$ and $\epsilon>0$.
    Assume that $\rho^*_{\bar u}\leq \rho_- < \rho_+$
    and $y_\nu(\bar t) < y(\bar t)$ (up to a subsequence)
    for every $\nu \in \N$.
    Then there exist a domain $\mathcal T_0$ and, for every $\nu \in \N$,
    a piecewise constant function $\xi_\nu(\cdot)$ and a time
    $t_f^{\xi_\nu} > \bar t$ such that $(t,\xi_\nu(t))\in \mathcal{T}_0$
    for every $t\in [\bar t,t_f^{\xi_\nu}[$, 
    and
    $\rho_\nu(t,x+)\in\, ]\rho_+-\epsilon,\rho_++\epsilon[$
    for every $(t,x)\in \left([\bar t,+\infty[\, \times  \,]\xi_\nu(t),
      +\infty[\right)\cap \mathcal{T}_0$.
      Moreover, there exist
      $c>0$ independent of $\nu$ and  $t_\nu\in[\bar t,\bar t+c[$ such that $y_\nu(t_\nu)=\xi_\nu(t_\nu)$ and $\lim_{\nu\to \infty} t_\nu=\bar t$;
       see Figure~\ref{lemma51}.
  \end{lemma}

  \begin{figure}[ht]
    \centering
    \includegraphics[width=0.45\linewidth]{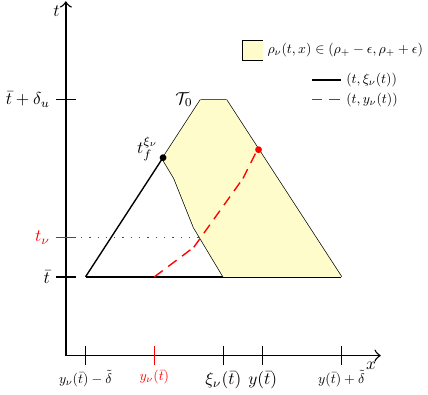}
    \caption{The situation of Lemma~\ref{oooo} in the case
      $\rho_{\bar u}^* \leq \rho_- <\rho_+ \leq R$ with
      $y_\nu(\bar t)< y(\bar t)$.
    }
    \label{lemma51}
  \end{figure}
 

  \begin{lemma}
    \label{rhon1}
    \textbf{\textup{\cite[Lemma 8]{LiardPiccoli2}}}.
    Let $\bar t \in \R^+ \setminus \mathcal{N}$ and $\epsilon>0$.
    Assume that $\rho_-,\rho_+\in [0,\rho^*_{\bar u}]$.
    Then there exists $\delta>0$ and  $\bar \nu \in \N$ such that
    \begin{equation*}
      \rho_\nu(\bar t,x)\in \,]0,\rho^*_{\bar u}+2\epsilon[
    \end{equation*}
    for every $\nu \ge \bar \nu$ and $x\in \,]\min\{y_\nu(\bar t),y(\bar t)\}
    -\delta,\max\{y_\nu(\bar t),y(\bar t)\}+ \delta[$.
\end{lemma}

\begin{lemma} \label{ine}
\textbf{\textup{\cite[Lemma 9]{LiardPiccoli2}}}.
If $\min\{\rho_+,\rho_-\} <\rho^*_{\bar u}< \max\{\rho_+,\rho_-\}$, then it holds $\rho_-<\rho^*_{\bar u}<\rho_+$.
\end{lemma}

 \begin{lemma}
    \label{ooooa} 
    \textbf{\textup{\cite[Lemma 10]{LiardPiccoli2}}}.
    Let $\bar t \in \R^+ \setminus \mathcal{N}$ and $\epsilon>0$.
    Assume that $\rho_-<\rho^*_{\bar u}<\rho_+$
    and $y_\nu(\bar t) < y(\bar t)$ (up to a subsequence)
    for every $\nu \in \N$.
    Then there exist a domain $\mathcal T_1$ and, for every $\nu \in \N$,
    a piecewise constant function $\xi^1_\nu(\cdot)$ and a time
    $t_f^{\xi_\nu^1} > \bar t$ such that $(t,\xi^1_\nu(t))\in \mathcal{T}_1$
    for every $t\in [\bar t,t_f^{\xi^1_\nu}[$, 
    and
    $\rho_\nu(t,x+)\in\, ]\rho_+-\epsilon,\rho_++\epsilon[$
    for every $(t,x)\in \left([\bar t,+\infty[\, \times  \,]\xi^1_\nu(t),
      +\infty[\right)\cap \mathcal{T}_1$.
      Moreover, there exist
      $c>0$ independent of $\nu$ and  $t_\nu\in[\bar t,\bar t+c[$ such that $y_\nu(t_\nu)=\xi^1_\nu(t_\nu)$ and $\lim_{\nu\to \infty} t_\nu=\bar t$.
  \end{lemma}

\section{The Riemann problem at the vehicle location}
\label{se:RP}


We recall here the definition of the solution to the Riemann problem, i.e. problem~\eqref{eq:control-model}
with initial data
\begin{equation}
  \label{eq:RP}
  y_0=0 \qquad\hbox{and}\qquad
  \rho_0(x) = \begin{cases}
        \rho_l & \textrm{ if } x < 0,
        \\
        \rho_r & \textrm{ if } x >  0.
\end{cases} 
\end{equation}
Denote by $\mathcal{R}$ the standard (i.e., without the constraint
\eqref{eq:constraint}) Riemann solver for
\eqref{eq:LWR}-\eqref{eq:RP}, i.e., the (right continuous) map
$(t,x)\mapsto \mathcal{R}(\rho_L,\rho_R)({x}/{t})$ given by the
standard weak entropy solution, see for instance~\cite{Lefloch_book}.
Moreover, given $u \in [0, V]$, let $\check{\rho}_u$ and $\hat{\rho}_u$, with
$\check{\rho}_u \leq \hat{\rho}_u$,
be the points defined in~\eqref{eq:check-hat}. \\
Following~\cite[Definition 3.1]{MR3264413}, the constrained Riemann
solver is defined as follows:
\begin{definition}
  \label{def:RS}
  The constrained Riemann solver
  $\mathcal{R}^u : [0,R]^2 \to \Ll1(\R;[0,R])$ is defined as
  follows.
  \begin{enumerate}
  \item If
    $f(\mathcal{R}(\rho_l,\rho_r)(u))> F_\alpha+u
    \,\mathcal{R}(\rho_l,\rho_r)(u)$, then
    \begin{equation*}
      \mathcal{R}^u(\rho_l,\rho_r)(x/t)= \left\{
        \begin{array}{ll}
          \mathcal{R}(\rho_l,\hat{\rho}_u)(x/t) 
          &\hbox{if}~x < u t,
            \vspace{.2cm}\\
          \mathcal{R}(\check{\rho}_u,\rho_r)(x/t) &\hbox{if}~x\geq u t,
        \end{array}
      \right.  \quad\hbox{and}\quad y(t)=u t.
    \end{equation*}
    
  \item If
    $f(\mathcal{R}(\rho_l,\rho_r)(u))\leq
    F_\alpha+u\,\mathcal{R}(\rho_L,\rho_R)(u)$, then
    \begin{equation*}
      \mathcal{R}^u (\rho_l,\rho_r)=\mathcal{R}(\rho_l,\rho_r)
      \quad\hbox{and}\quad y(t)=\omega(\rho_r;u)\, t.
    \end{equation*}
  \end{enumerate}
\end{definition}

In the next subsections we detail the structure of solutions in some cases.
%
%
\subsection{Case of a \emph{shock wave} in $\rho$: $\rho_l < \rho_r$.}
\label{se:rp:shock}
We suppose that $0 \le \rho_l < \rho_r \le R$.
Denote by $\sigma$ the speed
of the shock wave, i.e.
\begin{equation*}
  \sigma = \frac{f\left(\rho_l\right) - f\left(\rho_r\right)}{\rho_l - \rho_r}.
\end{equation*}
First assume $\sigma < u$.
The following possibilities hold.
\begin{enumerate}
\item $f\left(\rho_r\right) > \fhi_u\left(\rho_r\right)$.
  This hypothesis implies that $\check \rho_u < \rho_r < \hat \rho_u$, while
  $\sigma < u$ implies that the vehicle at $y$ enters in the region with
  density $\rho_r$. Moreover $\sigma < u$  and $\rho_l < \rho_r$
  imply that $\check \rho_u < \rho_l < \rho_r < \hat \rho_u$.
  Therefore the solution is given by
  \begin{equation*}
    \rho(t,x) = \left\{
      \begin{array}{ll}
        \rho_l 
        & \textrm{ if } x < \frac{f\left(\rho_l\right) 
          - f\left(\hat \rho_u\right)}{\rho_l - \hat \rho_u} t
        \vspace{.2cm}\\
        \hat \rho_u
        & \textrm{ if } \frac{f\left(\rho_l\right) 
          - f\left(\hat \rho_u\right)}{\rho_l - \hat \rho_u} t < x
          < ut
        \vspace{.2cm}\\
        \check \rho_u
        & \textrm{ if } ut < x < \frac{f\left(\rho_r\right) 
          - f\left(\check \rho_u\right)}{\rho_r - \check \rho_u} t
        \vspace{.2cm}\\
        \rho_r
        & \textrm{ if } x > \frac{f\left(\rho_r\right) 
          - f\left(\check \rho_u\right)}{\rho_r - \check \rho_u} t
      \end{array}
    \right.
  \end{equation*}
  and
  \begin{equation*}
    y(t) = y(0) + u t;
  \end{equation*}
  see Figure~\ref{fig:case-shock-a}.
  Finally note that $u \le v\left(\rho_l\right) \le v \left(\rho_r\right)$,
  so that the constraint in~(\ref{eq:RP}) for $u$ is satisfied.

\item $u \rho_r \le f\left(\rho_r\right) \le \fhi_u\left(\rho_r\right)$.
  This hypothesis implies that either $0 \le \rho_r \le \check \rho_u$
  or $\hat \rho_u < \rho_r < \rho_u^*$, while
  $\sigma < u$ implies that the vehicle at $y$ enters in the region with
  density $\rho_r$.
  If $0 \le \rho_r \le \check \rho_u$, then the assumptions 
  $\rho_l < \rho_r$ and $\sigma < u$ give a contradiction. Thus we deduce
  that $\hat \rho_u < \rho_r < \rho_u^*$.
  In this case the solution is given by
  \begin{equation*}
    \rho(t,x) = \left\{
      \begin{array}{ll}
        \rho_l 
        & \textrm{ if } x < \frac{f\left(\rho_l\right) 
          - f\left(\rho_r\right)}{\rho_l - \rho_r} t
        \vspace{.2cm}\\
        \rho_r
        & \textrm{ if } x > \frac{f\left(\rho_l\right) 
          - f\left(\rho_r\right)}{\rho_l - \rho_r} t
      \end{array}
    \right.
  \end{equation*}
  and
  \begin{equation*}
    y(t) = y(0) + u t;
  \end{equation*}
  see Figure~\ref{fig:case-shock-b}.
  Finally note that $u \le v\left(\rho_r\right)$,
  so that the constraint in~(\ref{eq:RP}) for $u$ is satisfied.

\item $f\left(\rho_r\right) < u \rho_r$.
  This hypothesis implies that $\rho_r > \rho_u^*$, while
  $\sigma < u$ implies that the vehicle at $y$ enters in the region with
  density $\rho_r$.
  In this case the solution is given by
  \begin{equation*}
    \rho(t,x) = \left\{
      \begin{array}{ll}
        \rho_l 
        & \textrm{ if } x < \frac{f\left(\rho_l\right) 
          - f\left(\rho_r\right)}{\rho_l - \rho_r} t
        \vspace{.2cm}\\
        \rho_r
        & \textrm{ if } x > \frac{f\left(\rho_l\right) 
          - f\left(\rho_r\right)}{\rho_l - \rho_r} t
      \end{array}
    \right.
  \end{equation*}
  and
  \begin{equation*}
    y(t) = y(0) + v(\rho_r) t,
  \end{equation*}
  see Figure~\ref{fig:case-shock-c}.
  Note that $u > v\left(\rho_l\right)$, so that the constraint
  in~\eqref{eq:RP}  is satisfied. The effective control
  is $u_e = v(\rho_r)$.
\end{enumerate}

\begin{figure}
  \centering
    \begin{tikzpicture}[line cap=round,line join=round,x=1.cm,y=0.7cm]

      \draw[<->] (0.2, 6.) -- (0.2, 1.) -- (5, 1.);

      \draw (0.2, 1.) to [out = 80, in = 180] (2.35, 4.5)
      to [out = 0, in = 100] (4.5, 1.);

      \draw (0.2, 1.) to [out = 80, in = 180] (1.77, 3.5)
      to [out = 0, in = 100] (3.5, 1.);

      \draw[very thick] (0.2, 1.) -- (0.2 + 4.5, 1. + 1.7);
      \draw[very thick] (0.2, 3) -- (0.2 + 4.5, 3. + 1.7);

      \draw[dashed] (4.2, 2.5) -- (4.2, 1.);
      \draw[dashed] (3.25, 4.17) -- (3.25, 1.);
      \draw[dashed] (0.78, 3.2) -- (0.78, 1.);

      \node[anchor=north,inner sep=0] at (0.2,0.9) {$0$};
      \node[anchor=north,inner sep=0] at (4.,0.9) {$\rho_u^*$};
      \node[anchor=north,inner sep=0] at (3.25,0.9) {$\hat\rho_u$};
      \node[anchor=north,inner sep=0] at (0.78,0.9) {$\check\rho_u$};
      \node[anchor=north,inner sep=0] at (5,0.9) {$\rho$};
      \node[anchor=east,inner sep=0] at (0.2,6) {$f$};
      \node[anchor=north,inner sep=0] at (4.5,0.9) {\small $R$};
      \node[anchor=south,inner sep=0] at (3.1,4.4) {$\rho_r$};
      \node[anchor=south,inner sep=0] at (1.1,4.) {$\rho_l$};

      \node[anchor=south east,inner sep=0] at (4.8,4.8) {$\fhi_u$};
      \node[anchor=south east,inner sep=0] at (4.8,2.8) {$u$};

      \draw[fill] (3., 4.31) circle (0.06);
      \draw[fill] (1.3, 4.) circle (0.06);
      \draw[very thick, dashed] (1.3, 4.) -- (3., 4.31);

      \draw[->] (6,1) -- (11, 1);
      \draw[->] (8.5, 1) -- (8.5, 6);

      \draw[color=black] (8.5, 1) -- (8.8, 6);
      \draw[very thick] (8.5, 1) -- (9.6, 6);
      \draw[color=black] (8.5, 1) -- (10.5, 6);

      \node[anchor=south,inner sep=0] at (8,4.) {$\rho_l$};
      \node[anchor=south,inner sep=0] at (10.2,4.) {$\rho_r$};

      \node[anchor=south,inner sep=0] at (9.,5.) {$\hat \rho_u$};
      \node[anchor=south,inner sep=0] at (9.8,5.) {$\check \rho_u$};

      \node[anchor=north,inner sep=0] at (11.,0.9) {$x$};
      \node[anchor=east,inner sep=0] at (8.4,6) {$t$};
      \node[anchor=north,inner sep=0] at (8.5,0.9) {$0$};

    \end{tikzpicture}    
  \caption{The Riemann problem. 
    The situation of a shock with speed less than $u$ and
    $f\left(\rho_r\right) > \fhi_u(\rho_r)$. In this case $y(t) = y(0) + ut$,
    while the solution for $\rho$ is composed by two classical shocks and
    one non classical shock.}
  \label{fig:case-shock-a}
\end{figure}
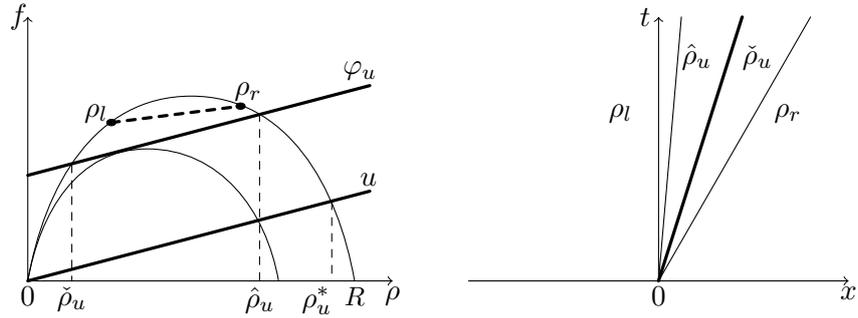

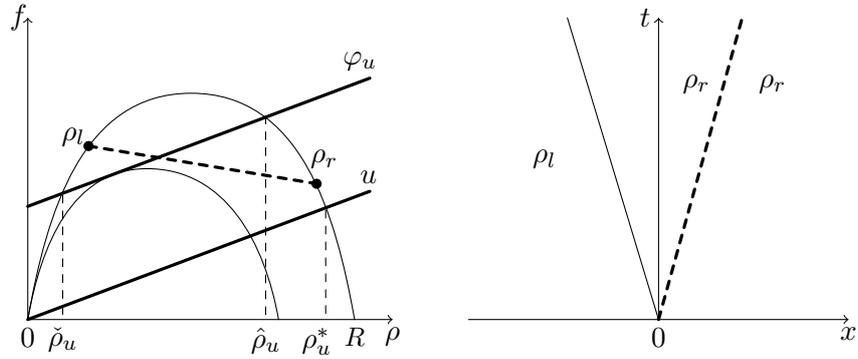
\begin{figure}
  \centering
    \begin{tikzpicture}[line cap=round,line join=round,x=1.cm,y=1cm]

      \draw[<->] (0.2, 5.) -- (0.2, 1.) -- (5, 1.);

      \draw (0.2, 1.) to [out = 80, in = 180] (2.35, 4.)
      to [out = 0, in = 100] (4.5, 1.);

      \draw (0.2, 1.) to [out = 80, in = 180] (1.77, 3.)
      to [out = 0, in = 100] (3.5, 1.);

      \draw[very thick] (0.2, 1.) -- (0.2 + 4.5, 1. + 1.7);
      \draw[very thick] (0.2, 2.5) -- (0.2 + 4.5, 2.5 + 1.7);

      \draw[dashed] (4.12, 2.45) -- (4.12, 1.);
      \draw[dashed] (3.33, 3.7) -- (3.33, 1.);
      \draw[dashed] (0.66, 2.65) -- (0.66, 1.);

      \node[anchor=north,inner sep=0] at (0.2,0.9) {$0$};
      \node[anchor=north,inner sep=0] at (4.,0.9) {$\rho_u^*$};
      \node[anchor=north,inner sep=0] at (3.33,0.9) {$\hat\rho_u$};
      \node[anchor=north,inner sep=0] at (0.66,0.9) {$\check\rho_u$};
      \node[anchor=north,inner sep=0] at (5,0.9) {$\rho$};
      \node[anchor=east,inner sep=0] at (0.2,5) {$f$};
      \node[anchor=north,inner sep=0] at (4.5,0.9) {\small $R$};
      \node[anchor=south,inner sep=0] at (4.1,3.) {$\rho_r$};
      \node[anchor=south,inner sep=0] at (0.8,3.3) {$\rho_l$};

      \node[anchor=south east,inner sep=0] at (4.8,4.3) {$\fhi_u$};
      \node[anchor=south east,inner sep=0] at (4.8,2.8) {$u$};

      \draw[ fill] (4., 2.8) circle (0.06);
      \draw[fill] (1., 3.3) circle (0.06);
      \draw[very thick, dashed] (1., 3.3) -- (4., 2.8);

      \draw[->] (6,1) -- (11, 1);
      \draw[->] (8.5, 1) -- (8.5, 5);

      \draw[color=black] (8.5, 1) -- (7.3, 5);
      \draw[very thick, dashed] (8.5, 1) -- (9.6, 5);

      \node[anchor=south,inner sep=0] at (7,3.) {$\rho_l$};
      \node[anchor=south,inner sep=0] at (10,4.) {$\rho_r$};
      \node[anchor=south,inner sep=0] at (9,4.) {$\rho_r$};


      \node[anchor=north,inner sep=0] at (11.,0.9) {$x$};
      \node[anchor=east,inner sep=0] at (8.4,5) {$t$};
      \node[anchor=north,inner sep=0] at (8.5,0.9) {$0$};

    \end{tikzpicture}    
  \caption{The Riemann problem. 
    The situation of a shock with speed less than $u$ and
    $u \rho_r \le f\left(\rho_r\right) \le \fhi_u\left(\rho_r\right)$.
    In this case $y(t) = y(0) + ut$,
    while the solution for $\rho$ is composed by the classical shock
    connecting $\rho_l$ to $\rho_r$.}
  \label{fig:case-shock-b}
\end{figure}

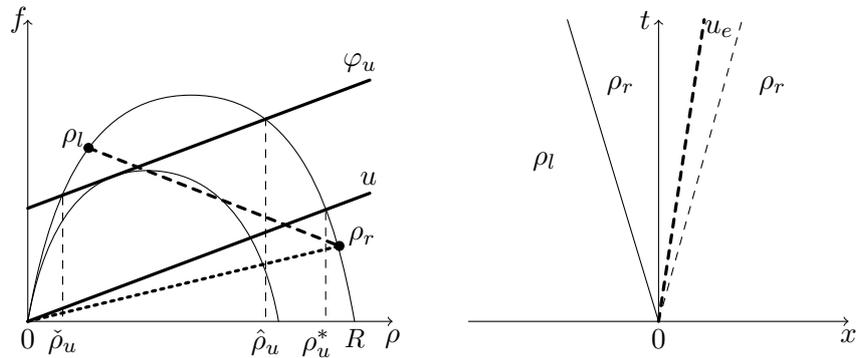
\begin{figure}
  \centering
    \begin{tikzpicture}[line cap=round,line join=round,x=1.cm,y=1cm]

      \draw[<->] (0.2, 5.) -- (0.2, 1.) -- (5, 1.);

      \draw (0.2, 1.) to [out = 80, in = 180] (2.35, 4.)
      to [out = 0, in = 100] (4.5, 1.);

      \draw (0.2, 1.) to [out = 80, in = 180] (1.77, 3.)
      to [out = 0, in = 100] (3.5, 1.);

      \draw[very thick] (0.2, 1.) -- (0.2 + 4.5, 1. + 1.7);
      \draw[very thick] (0.2, 2.5) -- (0.2 + 4.5, 2.5 + 1.7);

      \draw[dashed] (4.12, 2.45) -- (4.12, 1.);
      \draw[dashed] (3.33, 3.7) -- (3.33, 1.);
      \draw[dashed] (0.66, 2.65) -- (0.66, 1.);

      \node[anchor=north,inner sep=0] at (0.2,0.9) {$0$};
      \node[anchor=north,inner sep=0] at (4.,0.9) {$\rho_u^*$};
      \node[anchor=north,inner sep=0] at (3.33,0.9) {$\hat\rho_u$};
      \node[anchor=north,inner sep=0] at (0.66,0.9) {$\check\rho_u$};
      \node[anchor=north,inner sep=0] at (5,0.9) {$\rho$};
      \node[anchor=east,inner sep=0] at (0.2,5) {$f$};
      \node[anchor=north,inner sep=0] at (4.5,0.9) {\small $R$};
      \node[anchor=south,inner sep=0] at (4.6,2.) {$\rho_r$};
      \node[anchor=south,inner sep=0] at (0.8,3.3) {$\rho_l$};

      \node[anchor=south east,inner sep=0] at (4.8,4.3) {$\fhi_u$};
      \node[anchor=south east,inner sep=0] at (4.8,2.8) {$u$};

      \draw[fill] (4.3, 2.) circle (0.06);
      \draw[fill] (1., 3.3) circle (0.06);
      \draw[very thick, dashed] (1., 3.3) -- (4.3, 2.);
      \draw[very thick, dotted] (0.2, 1) -- (4.3, 2.); 

      \draw[->] (6,1) -- (11, 1);
      \draw[->] (8.5, 1) -- (8.5, 5);

      \draw[color=black] (8.5, 1) -- (7.3, 5);
      \draw[dashed] (8.5, 1) -- (9.6, 5);
      \draw[very thick, dashed] (8.5, 1) -- (9.1, 5);

      \node[anchor=south,inner sep=0] at (7,3.) {$\rho_l$};
      \node[anchor=south,inner sep=0] at (10,4.) {$\rho_r$};
      \node[anchor=south,inner sep=0] at (8.,4.) {$\rho_r$};


      \node[anchor=north,inner sep=0] at (11.,0.9) {$x$};
      \node[anchor=east,inner sep=0] at (8.4,5) {$t$};
      \node[anchor=north,inner sep=0] at (8.5,0.9) {$0$};
      \node[anchor=north west,inner sep=0] at (9.1,5) {$u_{e}$};

    \end{tikzpicture}    
  \caption{The Riemann problem. 
    The situation of a shock with speed less than $u$ and
    $f\left(\rho_r\right) < u \rho_r$.
    In this case $y(t) = y(0) + u_et$ with $u_e = v\left(\rho_r\right)$,
    while the solution for $\rho$ is composed by the classical shock
    connecting $\rho_l$ to $\rho_r$.}
  \label{fig:case-shock-c}
\end{figure}

Now suppose $\sigma > u$.
The following possibilities hold.
\begin{enumerate}
\item $f\left(\rho_l\right) > \fhi_u\left(\rho_l\right)$.
  This hypothesis implies that $\check \rho_u < \rho_l < \hat \rho_u$, while
  $\sigma > u$ implies that the vehicle at $y$ enters in the region with
  density $\rho_l$. Moreover $\sigma > u$  and $\rho_l < \rho_r$
  imply that $\check \rho_u < \rho_l < \rho_r < \hat \rho_u$.
  In this case the solution is given by
  \begin{equation*}
    \rho(t,x) = \left\{
      \begin{array}{ll}
        \rho_l 
        & \textrm{ if } x < \frac{f\left(\rho_l\right) 
          - f\left(\hat \rho_u\right)}{\rho_l - \hat \rho_u} t
        \vspace{.2cm}\\
        \hat \rho_u
        & \textrm{ if } \frac{f\left(\rho_l\right) 
          - f\left(\hat \rho_u\right)}{\rho_l - \hat \rho_u} t < x
          < ut
        \vspace{.2cm}\\
        \check \rho_u
        & \textrm{ if } ut < x < \frac{f\left(\rho_r\right) 
          - f\left(\check \rho_u\right)}{\rho_r - \check \rho_u} t
        \vspace{.2cm}\\
        \rho_r
        & \textrm{ if } x > \frac{f\left(\rho_r\right) 
          - f\left(\check \rho_u\right)}{\rho_r - \check \rho_u} t
      \end{array}
    \right.
  \end{equation*}
  and
  \begin{equation*}
    y(t) = y(0) + u t;
  \end{equation*}
  see Figure~\ref{fig:case-shock-II-a}.
  Finally note that $u \le v\left(\rho_l\right)$,
  so that the constraint in~(\ref{eq:RP}) for $u$ is satisfied.

\item $u \rho_l \le f\left(\rho_l\right) \le \fhi_u\left(\rho_l\right)$.
  This hypothesis implies that either $0 \le \rho_l \le \check \rho_u$
  or $\hat \rho_u \le \rho_l \le \rho_u^*$, while
  $\sigma > u$ implies that the vehicle at $y$ enters in the region with
  density $\rho_l$.
  If $\hat \rho_u \le \rho_l \le \rho_u^*$, then the assumptions 
  $\rho_l < \rho_r$ and $\sigma > u$ produce a contradiction. Thus we deduce
  that $0 \le \rho_l \le \check \rho_u$.
  In this case the solution is given by
  \begin{equation*}
    \rho(t,x) = \left\{
      \begin{array}{ll}
        \rho_l 
        & \textrm{ if } x < \frac{f\left(\rho_l\right) 
          - f\left(\rho_r\right)}{\rho_l - \rho_r} t
        \vspace{.2cm}\\
        \rho_r
        & \textrm{ if } x > \frac{f\left(\rho_l\right) 
          - f\left(\rho_r\right)}{\rho_l - \rho_r} t
      \end{array}
    \right.
  \end{equation*}
  and
  \begin{equation*}
    y(t) = y(0) + u t;
  \end{equation*}
  see Figure~\ref{fig:case-shock-II-b}.
  Finally note that $u \le v\left(\rho_l\right)$,
  so that the constraint in~(\ref{eq:RP}) for $u$ is satisfied.

\item $f\left(\rho_l\right) < u \rho_l$.
  This hypothesis implies that $\rho_l > \rho_u^*$, while
  $\sigma > u$ implies that the vehicle at $y$ enters in the region with
  density $\rho_l$.
  The fact that $\rho_l < \rho_r$ is in contradiction with 
  $\sigma > u$, so that this case does not happen.
\end{enumerate}

\begin{figure}
  \centering
    \begin{tikzpicture}[line cap=round,line join=round,x=1.cm,y=0.7cm]

      \draw[<->] (0.2, 6.) -- (0.2, 1.) -- (5, 1.);

      \draw (0.2, 1.) to [out = 80, in = 180] (2.35, 4.5)
      to [out = 0, in = 100] (4.5, 1.);

      \draw (0.2, 1.) to [out = 80, in = 180] (1.77, 3.5)
      to [out = 0, in = 100] (3.5, 1.);

      \draw[very thick] (0.2, 1.) -- (0.2 + 4.5, 1. + 1.7);
      \draw[very thick] (0.2, 3) -- (0.2 + 4.5, 3. + 1.7);

      \draw[dashed] (4.2, 2.5) -- (4.2, 1.);
      \draw[dashed] (3.25, 4.17) -- (3.25, 1.);
      \draw[dashed] (0.78, 3.2) -- (0.78, 1.);

      \node[anchor=north,inner sep=0] at (0.2,0.9) {$0$};
      \node[anchor=north,inner sep=0] at (4.,0.9) {$\rho_u^*$};
      \node[anchor=north,inner sep=0] at (3.25,0.9) {$\hat\rho_u$};
      \node[anchor=north,inner sep=0] at (0.78,0.9) {$\check\rho_u$};
      \node[anchor=north,inner sep=0] at (5,0.9) {$\rho$};
      \node[anchor=east,inner sep=0] at (0.2,6) {$f$};
      \node[anchor=north,inner sep=0] at (4.5,0.9) {\small $R$};
      \node[anchor=south,inner sep=0] at (2.5,4.6) {$\rho_r$};
      \node[anchor=south,inner sep=0] at (0.8,3.6) {$\rho_l$};

      \node[anchor=south east,inner sep=0] at (4.8,4.8) {$\fhi_u$};
      \node[anchor=south east,inner sep=0] at (4.8,2.8) {$u$};

      \draw[fill] (2.4, 4.5) circle (0.06);
      \draw[fill] (1., 3.6) circle (0.06);
      \draw[very thick, dashed] (1., 3.6) -- (2.4, 4.5);

      \draw[->] (6,1) -- (11, 1);
      \draw[->] (8.5, 1) -- (8.5, 6);

      \draw[color=black] (8.5, 1) -- (8.9, 6);
      \draw[very thick] (8.5, 1) -- (9.6, 6);
      \draw[color=black] (8.5, 1) -- (10.3, 6);

      \node[anchor=south,inner sep=0] at (8,4.) {$\rho_l$};
      \node[anchor=south,inner sep=0] at (10,4.) {$\rho_r$};

      \node[anchor=south,inner sep=0] at (9.1,5.)
      {$\hat \rho_u$};
      \node[anchor=south,inner sep=0] at (9.8,5.)
      {$\check \rho_u$};

      \node[anchor=north,inner sep=0] at (11.,0.9) { $x$};
      \node[anchor=east,inner sep=0] at (8.4,6) {$t$};
      \node[anchor=north,inner sep=0] at (8.5,0.9) {$0$};

    \end{tikzpicture}    
  \caption{The Riemann problem. 
    The situation of a shock with speed bigger than $u$ and
    $f\left(\rho_l\right) > \fhi_u(\rho_l)$. In this case $y(t) = y(0) + ut$,
    while the solution for $\rho$ is composed by two classical shocks and
    one non classical shock.}
  \label{fig:case-shock-II-a}
\end{figure}
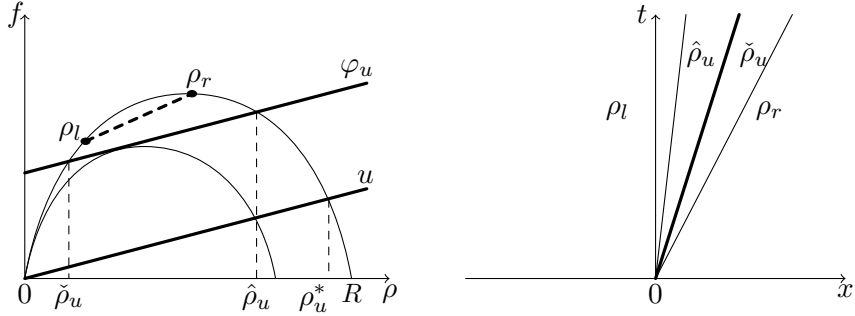

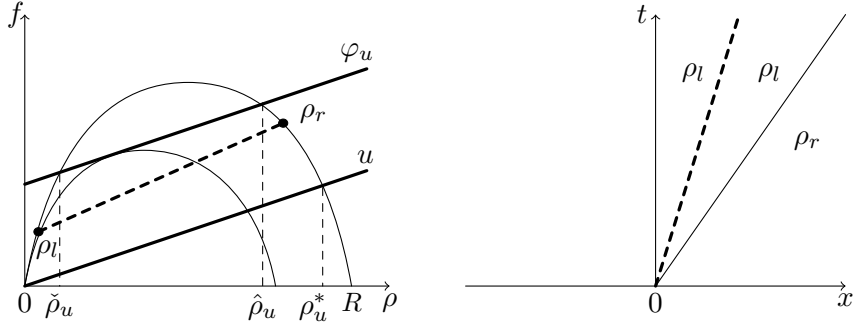
\begin{figure}
  \centering
    \begin{tikzpicture}[line cap=round,line join=round,x=1.cm,y=0.9cm]

      \draw[<->] (0.2, 5.) -- (0.2, 1.) -- (5, 1.);

      \draw (0.2, 1.) to [out = 80, in = 180] (2.35, 4.)
      to [out = 0, in = 100] (4.5, 1.);

      \draw (0.2, 1.) to [out = 80, in = 180] (1.77, 3.)
      to [out = 0, in = 100] (3.5, 1.);

      \draw[very thick] (0.2, 1.) -- (0.2 + 4.5, 1. + 1.7);
      \draw[very thick] (0.2, 2.5) -- (0.2 + 4.5, 2.5 + 1.7);

      \draw[dashed] (4.12, 2.45) -- (4.12, 1.);
      \draw[dashed] (3.33, 3.7) -- (3.33, 1.);
      \draw[dashed] (0.66, 2.65) -- (0.66, 1.);

      \node[anchor=north,inner sep=0] at (0.2,0.9) {$0$};
      \node[anchor=north,inner sep=0] at (4.,0.9) {$\rho_u^*$};
      \node[anchor=north,inner sep=0] at (3.33,0.9) {$\hat\rho_u$};
      \node[anchor=north,inner sep=0] at (0.66,0.9) {$\check\rho_u$};
      \node[anchor=north,inner sep=0] at (5,0.9) {$\rho$};
      \node[anchor=east,inner sep=0] at (0.2,5) {$f$};
      \node[anchor=north,inner sep=0] at (4.5,0.9) {\small $R$};
      \node[anchor=south,inner sep=0] at (4.,3.4) {$\rho_r$};
      \node[anchor=south,inner sep=0] at (0.5,1.4) {$\rho_l$};

      \node[anchor=south east,inner sep=0] at (4.8,4.3) {$\fhi_u$};
      \node[anchor=south east,inner sep=0] at (4.8,2.8) {$u$};

      \draw[fill] (3.6, 3.4) circle (0.06);
      \draw[fill] (0.38, 1.8) circle (0.06);
      \draw[very thick, dashed] (.38, 1.8) -- (3.6, 3.4);

      \draw[->] (6,1) -- (11, 1);
      \draw[->] (8.5, 1) -- (8.5, 5);

      \draw[color=black] (8.5, 1) -- (11., 5);
      \draw[very thick, dashed] (8.5, 1) -- (9.6, 5);

      \node[anchor=south,inner sep=0] at (10.5,3.) {$\rho_r$};
      \node[anchor=south,inner sep=0] at (10,4.) {$\rho_l$};
      \node[anchor=south,inner sep=0] at (9,4.) {$\rho_l$};


      \node[anchor=north,inner sep=0] at (11.,0.9) {$x$};
      \node[anchor=east,inner sep=0] at (8.4,5) {$t$};
      \node[anchor=north,inner sep=0] at (8.5,0.9) {$0$};

    \end{tikzpicture}    
  \caption{The Riemann problem. 
    The situation of a shock with speed greater than $u$ and
    $u \rho_l \le f\left(\rho_l\right) \le \fhi_u\left(\rho_l\right)$.
    In this case $y(t) = y(0) + ut$,
    while the solution for $\rho$ is composed by the classical shock
    connecting $\rho_l$ to $\rho_r$.}
  \label{fig:case-shock-II-b}
\end{figure}

Finally suppose $\sigma = u$.
The following possibilities hold.
\begin{enumerate}
\item $f\left(\rho_r\right) > \fhi_u\left(\rho_r\right)$.
  This hypothesis implies that $\check \rho_u < \rho_r < \hat \rho_u$, while
  $\sigma = u$ implies that the vehicle at $y$ has density $\rho_r$ in front.
  Moreover $\sigma = u$  and $\rho_l < \rho_r$
  imply that $\check \rho_u < \rho_l < \rho_r < \hat \rho_u$.
%
%
  In this case the solution is given by
  \begin{equation*}
    \rho(t,x) = \left\{
      \begin{array}{ll}
        \rho_l 
        & \textrm{ if } x < \frac{f\left(\rho_l\right) 
          - f\left(\hat \rho_u\right)}{\rho_l - \hat \rho_u} t
        \vspace{.2cm}\\
        \hat \rho_u
        & \textrm{ if } \frac{f\left(\rho_l\right) 
          - f\left(\hat \rho_u\right)}{\rho_l - \hat \rho_u} t < x
          < ut
        \vspace{.2cm}\\
        \check \rho_u
        & \textrm{ if } ut < x < \frac{f\left(\rho_r\right) 
          - f\left(\check \rho_u\right)}{\rho_r - \check \rho_u} t
        \vspace{.2cm}\\
        \rho_r
        & \textrm{ if } x > \frac{f\left(\rho_r\right) 
          - f\left(\check \rho_u\right)}{\rho_r - \check \rho_u} t
      \end{array}
    \right.
  \end{equation*}
  and
  \begin{equation*}
    y(t) = y(0) + u t.
  \end{equation*}
  Finally note that $u \le v\left(\rho_r\right)$,
  so that the constraint in~(\ref{eq:RP}) for $u$ is satisfied.

\item $u \rho_r \le f\left(\rho_r\right) \le \fhi_u\left(\rho_r\right)$.
  This hypothesis implies that either $0 \le \rho_r \le \check \rho_u$
  or $\hat \rho_u \le \rho_r \le \rho_u^*$, while
  $\sigma = u$ implies that in front of the vehicle at $y$
  there is density $\rho_r$.
  If $0 \le \rho_r \le \check \rho_u$, then the assumptions 
  $\rho_l < \rho_r$ and $\sigma = u$ produce a contradiction. Thus we deduce
  that $\hat \rho_u \le \rho_r \le \rho_u^*$.
  In this case the solution is given by
  \begin{equation*}
    \rho(t,x) = \left\{
      \begin{array}{ll}
        \rho_l 
        & \textrm{ if } x < \frac{f\left(\rho_l\right) 
          - f\left(\rho_r\right)}{\rho_l - \rho_r} t
        \vspace{.2cm}\\
        \rho_r
        & \textrm{ if } x > \frac{f\left(\rho_l\right) 
          - f\left(\rho_r\right)}{\rho_l - \rho_r} t
      \end{array}
    \right.
  \end{equation*}
  and
  \begin{equation*}
    y(t) = y(0) + u t.
  \end{equation*}
  Finally note that $u \le v\left(\rho_r\right)$,
  so that the constraint in~(\ref{eq:RP}) for $u$ is satisfied.

\item $f\left(\rho_r\right) < u \rho_r$.
  This hypothesis implies that $\rho_r > \rho_u$, while
  $\sigma = u$ implies that in the front of the vehicle at $y$
  there is density $\rho_r$.
  The fact that $\rho_l < \rho_r$ is in contradiction with 
  $\sigma = u$, so that this case does not happen.
\end{enumerate}

%
%
\subsection{Case of a \emph{rarefaction wave}
  in $\rho$: $\rho_r < \rho_l$.}
\label{se:rp:rarefaction}
Suppose that $0 \le \rho_r < \rho_l \le R$. Denote with $\sigma_l$
and $\sigma_r$ respectively the characteristic speeds of $\rho_l$ and
$\rho_r$, i.e.
\begin{equation*}
  \sigma_l = f'\left(\rho_l\right)
  \qquad \textrm{ and }\qquad
  \sigma_r = f'\left(\rho_r\right).
\end{equation*}
Note that $\sigma_l < \sigma_r$.

First assume that $\sigma_r < u$, so that
the density in front of the vehicle at $y$ is $\rho_r$.
The following possibilities hold.
\begin{enumerate}
\item $f\left(\rho_r\right) > \fhi_u\left(\rho_r\right)$.
  This hypothesis implies that $\check \rho_u < \rho_r < \hat \rho_u$.
  If $\rho_l \le \hat \rho_u$, then the solution for $\rho$ is given by
  \begin{equation*}
    \rho(t,x) = \left\{
      \begin{array}{ll}
        \rho_l 
        & \textrm{ if } x < \frac{f\left(\rho_l\right) 
          - f\left(\hat \rho_u\right)}{\rho_l - \hat \rho_u} t
        \vspace{.2cm}\\
        \hat \rho_u
        & \textrm{ if } \frac{f\left(\rho_l\right) 
          - f\left(\hat \rho_u\right)}{\rho_l - \hat \rho_u} t < x
          < ut
        \vspace{.2cm}\\
        \check \rho_u
        & \textrm{ if } ut < x < \frac{f\left(\rho_r\right) 
          - f\left(\check \rho_u\right)}{\rho_r - \check \rho_u} t
        \vspace{.2cm}\\
        \rho_r
        & \textrm{ if } x > \frac{f\left(\rho_r\right) 
          - f\left(\check \rho_u\right)}{\rho_r - \check \rho_u} t,
      \end{array}
    \right.
  \end{equation*}
  while if $\rho_l > \hat \rho_u$, then the solution for $\rho$ is given by
  \begin{equation*}
    \rho(t,x) = \left\{
      \begin{array}{ll}
        \rho_l 
        & \textrm{ if } x < f'(\rho_l) t
        \vspace{.2cm}\\
        \left(f'\right)^{-1}\left(\frac{x}{t}\right)
        & 
          \textrm{ if }  f'(\rho_l) t < x
          < f'\left(\hat \rho_u\right) t
        \vspace{.2cm}\\
        \hat \rho_u
        & \textrm{ if } f'\left(\hat \rho_u\right) t < x
          < ut
        \vspace{.2cm}\\
        \check \rho_u
        & \textrm{ if } ut < x < \frac{f\left(\rho_r\right) 
          - f\left(\check \rho_u\right)}{\rho_r - \check \rho_u} t
        \vspace{.2cm}\\
        \rho_r
        & \textrm{ if } x > \frac{f\left(\rho_r\right) 
          - f\left(\check \rho_u\right)}{\rho_r - \check \rho_u} t.
      \end{array}
    \right.
  \end{equation*}
  The solution for $y$ is given by
  \begin{equation*}
    y(t) = u t;
  \end{equation*}
  see Figure~\ref{fig:case-rarefaction-a}.
  Finally note that $u \le v\left(\check \rho_u\right)$,
  so that the constraint in~(\ref{eq:RP}) for $u$ is satisfied.
  \begin{figure}
    \centering
    \begin{tikzpicture}[line cap=round,line join=round,x=1.cm,y=0.7cm]

      \draw[<->] (0.2, 6.) -- (0.2, 1.) -- (5, 1.);

      \draw (0.2, 1.) to [out = 80, in = 180] (2.35, 4.5) to [out = 0,
      in = 100] (4.5, 1.);

      \draw (0.2, 1.) to [out = 80, in = 180] (1.77, 3.5) to [out = 0,
      in = 100] (3.5, 1.);

      \draw[very thick] (0.2, 1.) -- (0.2 + 4.5,1. + 1.7); 
      \draw[very thick] (0.2, 3) -- (0.2 + 4.5, 3. + 1.7);

      \draw[dashed] (4.2, 2.5) -- (4.2, 1.); \draw[dashed] (3.25,
      4.17) -- (3.25, 1.); \draw[dashed] (0.78, 3.2) -- (0.78, 1.);

      \node[anchor=north,inner sep=0] at (0.2,0.9) {$0$};
      \node[anchor=north,inner sep=0] at (4.,0.9) {$\rho_u^*$};
      \node[anchor=north,inner sep=0] at (3.25,0.9) {$\hat\rho_u$};
      \node[anchor=north,inner sep=0] at (0.78,0.9) {$\check\rho_u$};
      \node[anchor=north,inner sep=0] at (5,0.9) {$\rho$};
      \node[anchor=east,inner sep=0] at (0.2,6) {$f$};
      \node[anchor=north,inner sep=0] at (4.5,0.9) {\small $R$};
      \node[anchor=south,inner sep=0] at (3.1,4.4) {$\rho_r$};
      \node[anchor=south,inner sep=0] at (4.2,3.1) {$\rho_l$};

      \node[anchor=south east,inner sep=0] at (4.8,4.8) {$\fhi_u$};
      \node[anchor=south east,inner sep=0] at (4.8,2.8) {$u$};

      \draw[fill] (3., 4.31) circle (0.06);
      \draw[fill] (4., 3.01) circle (0.06);
      \draw[very thick, dashed] (4, 3.01) to [out=120, in=340]
      (3., 4.31);

      \draw[->] (6,1) -- (11, 1);
      \draw[->] (8.5, 1) -- (8.5, 6);

      \draw[color=black, dashed] (8.5, 1) -- (7, 6);
      \draw[color=black, dashed] (8.5, 1) -- (7.1, 6);
      \draw[color=black, dashed] (8.5, 1) -- (7.2, 6);
      \draw[color=black, dashed] (8.5, 1) -- (7.3, 6);

      \draw[very thick] (8.5, 1) -- (9.6, 6); 
      \draw[color=black] (8.5, 1) -- (10.1, 6);

      \node[anchor=south,inner sep=0] at (7,4.) {$\rho_l$};
      \node[anchor=south,inner sep=0] at (10,4.) {$\rho_r$};

      \node[anchor=south,inner sep=0] at (9.,5.)  {$\hat \rho_u$};
      \node[anchor=south,inner sep=0] at (8.1,5.)  {$\hat \rho_u$};

      \node[anchor=south,inner sep=0] at (9.65,5.) {$\check \rho_u$};

      \node[anchor=north,inner sep=0] at (11.,0.9) {$x$}; 
      \node[anchor=east,inner sep=0] at (8.4,6) {$t$}; 
      \node[anchor=north,inner sep=0] at (8.5,0.9) {$0$};

    \end{tikzpicture}
    \caption{The Riemann problem.  The situation of a rarefaction with speed
      less than $u$ and $f\left(\rho_r\right) > \fhi_u(\rho_r)$. In
      this case $y(t) = y(0) + ut$, while the solution for $\rho$ is
      composed by two classical waves and one non classical shock.}
    \label{fig:case-rarefaction-a}
  \end{figure}
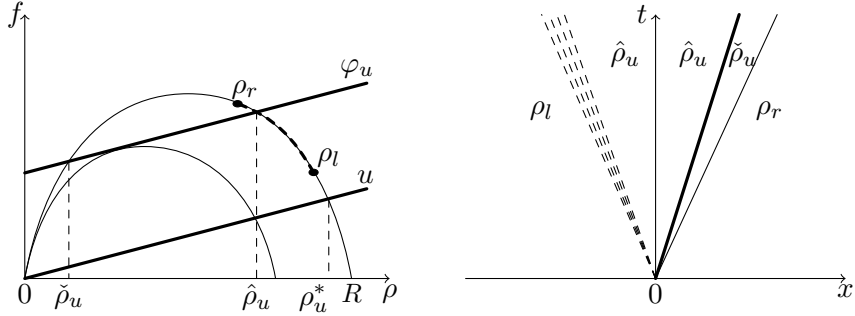

\item $u \rho_r \le f\left(\rho_r\right) \le \fhi_u\left(\rho_r\right)$.
  This hypothesis implies that either $0 \le \rho_r \le \check \rho_u$
  or $\hat \rho_u \le \rho_r \le \rho_u^*$
  If $0 \le \rho_r \le \check \rho_u$, then the assumption 
  $\sigma_r < u$ gives a contradiction. Thus we deduce
  that $\hat \rho_u \le \rho_r \le \rho_u^*$.
  In this case the solution is given by
  \begin{equation*}
    \rho(t,x) = \left\{
      \begin{array}{ll}
        \rho_l 
        & \textrm{ if } x < f'\left(\rho_l\right) t
        \vspace{.2cm}\\
        \left(f'\right)^{-1} \left(\frac{x}{t}\right)
        &
          \textrm{ if } f'\left(\rho_l\right) t < x < f'\left(\rho_r\right) t 
        \vspace{.2cm}\\
        \rho_r
        & \textrm{ if } x > f'\left(\rho_r\right) t
      \end{array}
    \right.
  \end{equation*}
  and
  \begin{equation*}
    y(t) = u t;
  \end{equation*}
  see Figure~\ref{fig:case-rarefaction-b}.
  Finally note that $u \le v\left(\rho_r\right)$,
  so that the constraint in~(\ref{eq:RP}) for $u$ is satisfied.
  \begin{figure}
    \centering
    \begin{tikzpicture}[line cap=round,line join=round,x=1.cm,y=0.7cm]

      \draw[<->] (0.2, 6.) -- (0.2, 1.) -- (5, 1.);

      \draw (0.2, 1.) to [out = 80, in = 180] (2.35, 4.5) to [out = 0,
      in = 100] (4.5, 1.);

      \draw (0.2, 1.) to [out = 80, in = 180] (1.77, 3.5) to [out = 0,
      in = 100] (3.5, 1.);

      \draw[very thick] (0.2, 1.) -- (0.2 + 4.5,1. + 1.7); 
      \draw[very thick] (0.2, 3) -- (0.2 + 4.5, 3. + 1.7);

      \draw[dashed] (4.2, 2.5) -- (4.2, 1.); \draw[dashed] (3.25,
      4.17) -- (3.25, 1.); \draw[dashed] (0.78, 3.2) -- (0.78, 1.);

      \node[anchor=north,inner sep=0] at (0.2,0.9) {$0$};
      \node[anchor=north,inner sep=0] at (4.,0.9) {$\rho_u^*$};
      \node[anchor=north,inner sep=0] at (3.25,0.9) {$\hat\rho_u$};
      \node[anchor=north,inner sep=0] at (0.78,0.9) {$\check\rho_u$};
      \node[anchor=north,inner sep=0] at (5,0.9) {$\rho$};
      \node[anchor=east,inner sep=0] at (0.2,6) {$f$};
      \node[anchor=north,inner sep=0] at (4.5,0.9) {\small $R$};
      \node[anchor=south,inner sep=0] at (3.8,3.9) {$\rho_r$};
      \node[anchor=south,inner sep=0] at (4.2,3.1) {$\rho_l$};

      \node[anchor=south east,inner sep=0] at (4.8,4.8) {$\fhi_u$};
      \node[anchor=south east,inner sep=0] at (4.8,2.8) {$u$};

      \draw[fill] (3.5, 3.9) circle (0.06);
      \draw[fill] (4., 3.01) circle (0.06);
      \draw[very thick, dashed] (4, 3.01) to [out=120, in=320]
      (3.5, 3.9);

      \draw[->] (6,1) -- (11, 1);
      \draw[->] (8.5, 1) -- (8.5, 6);

      \draw[dashed] (8.5, 1) -- (7, 6);
      \draw[dashed] (8.5, 1) -- (7.1, 6);
      \draw[dashed] (8.5, 1) -- (7.2, 6);
      \draw[dashed] (8.5, 1) -- (7.3, 6);

      \draw[very thick, dashed] (8.5, 1) -- (9.6, 6); 

      \node[anchor=south,inner sep=0] at (7,4.) {$\rho_l$};
      \node[anchor=south,inner sep=0] at (10,4.) {$\rho_r$};

      \node[anchor=south,inner sep=0] at (9.,5.)  {$\rho_r$};
      \node[anchor=south,inner sep=0] at (8.1,5.)  {$\rho_r$};

      \node[anchor=north,inner sep=0] at (11.,0.9) {$x$}; 
      \node[anchor=east,inner sep=0] at (8.4,6) {$t$}; 
      \node[anchor=north,inner sep=0] at (8.5,0.9) {$0$};

    \end{tikzpicture}
    \caption{The Riemann problem.  The situation of a rarefaction with speed
      less than $u$ and $u \rho_r \le f\left(\rho_r\right) \le \fhi_u(\rho_r)$.
      In
      this case $y(t) = ut$, while the solution for $\rho$ is
      composed by a classical rarefaction wave.}
    \label{fig:case-rarefaction-b}
  \end{figure}

\item $f\left(\rho_r\right) < u \rho_r$.
  This hypothesis implies that $\rho_r > \rho_u^*$.
  In this case the solution is given by
  \begin{equation*}
    \rho(t,x) = \left\{
      \begin{array}{ll}
        \rho_l 
        & \textrm{ if } x < f'\left(\rho_l\right) t
        \vspace{.2cm}\\
        \left(f'\right)^{-1} \left(\frac{x}{t}\right)
        &
          \textrm{ if } f'\left(\rho_l\right) t < x < f'\left(\rho_r\right) t 
        \vspace{.2cm}\\
        \rho_r
        & \textrm{ if } x > f'\left(\rho_r\right) t
      \end{array}
    \right.
  \end{equation*}
  and
  \begin{equation*}
    y(t) = u_{e} t,
  \end{equation*}
  where $u_e = v(\rho_r) < u$;
  see Figure~\ref{fig:case-rarefaction-c}.
  Note that $u > v\left(\rho_l\right)$, so that the constraint
  in~(\ref{eq:RP}) for $u$ is not satisfied. Thus the effective control
  is $u_e = v(\rho_r)$.
  \begin{figure}
    \centering
    \begin{tikzpicture}[line cap=round,line join=round,x=1.cm,y=0.7cm]
      
      \draw[<->] (0.2, 6.) -- (0.2, 1.) -- (5, 1.);

      \draw (0.2, 1.) to [out = 80, in = 180] (2.35, 4.5) to [out = 0,
      in = 100] (4.5, 1.);

      \draw (0.2, 1.) to [out = 80, in = 180] (1.77, 3.5) to [out = 0,
      in = 100] (3.5, 1.);

      \draw[very thick] (0.2, 1.) -- (0.2 + 4.5,1. + 1.7); 
      \draw[very thick] (0.2, 3) --  (0.2 + 4.5, 3. + 1.7);

      \draw[very thick, dotted] (0.2, 1.) -- (4.28, 2.2);

      \draw[dashed] (4.2, 2.5) -- (4.2, 1.); \draw[dashed] (3.25,
      4.17) -- (3.25, 1.); \draw[dashed] (0.78, 3.2) -- (0.78, 1.);

      \node[anchor=north,inner sep=0] at (0.2,0.9) {$0$};
      \node[anchor=north,inner sep=0] at (4.,0.9) {$\rho_u^*$};
      \node[anchor=north,inner sep=0] at (3.25,0.9) {$\hat\rho_u$};
      \node[anchor=north,inner sep=0] at (0.78,0.9) {$\check\rho_u$};
      \node[anchor=north,inner sep=0] at (5,0.9) {$\rho$};
      \node[anchor=east,inner sep=0] at (0.2,6) {$f$};
      \node[anchor=north,inner sep=0] at (4.5,0.9) {\small $R$};
      \node[anchor=south,inner sep=0] at (4.6,2.) {$\rho_r$};
      \node[anchor=south,inner sep=0] at (4.7,1.6) {$\rho_l$};

      \node[anchor=south east,inner sep=0] at (4.8,4.8) {$\fhi_u$};
      \node[anchor=south east,inner sep=0] at (4.8,2.8) {$u$};

      \draw[fill] (4.28, 2.2) circle (0.06);
      \draw[fill] (4.4, 1.7) circle (0.06);
      \draw[very thick, dashed] (4.4, 1.7) to [out=125, in=305]
      (4.28, 2.2);

      \draw[->] (6,1) -- (11, 1);
      \draw[->] (8.5, 1) -- (8.5, 6);

      \draw[dashed] (8.5, 1) -- (7, 6);
      \draw[dashed] (8.5, 1) -- (7.1, 6);
      \draw[dashed] (8.5, 1) -- (7.2, 6);
      \draw[dashed] (8.5, 1) -- (7.3, 6);

      \draw[color=black] (8.5, 1) -- (7.3, 6); 
      \draw[very thick, dashed] (8.5, 1) -- (9.3, 6); 

      \node[anchor=south,inner sep=0] at (7,4.) {$\rho_l$};

      \node[anchor=south,inner sep=0] at (9.,5.)  {$\rho_r$};
      \node[anchor=south,inner sep=0] at (8.1,5.)  {$\rho_r$};

      \node[anchor=north,inner sep=0] at (11.,0.9) {$x$}; 
      \node[anchor=east,inner sep=0] at (8.4,6) {$t$}; 
      \node[anchor=north,inner sep=0] at (8.5,0.9) {$0$};

    \end{tikzpicture}
    \caption{The Riemann problem.  The situation of a rarefaction with speed
      less than $u$ and $f\left(\rho_r\right) < u \rho_r$.
      In
      this case $y(t) = u_e t$, while the solution for $\rho$ is
      composed by a classical rarefaction wave.}
    \label{fig:case-rarefaction-c}
  \end{figure}
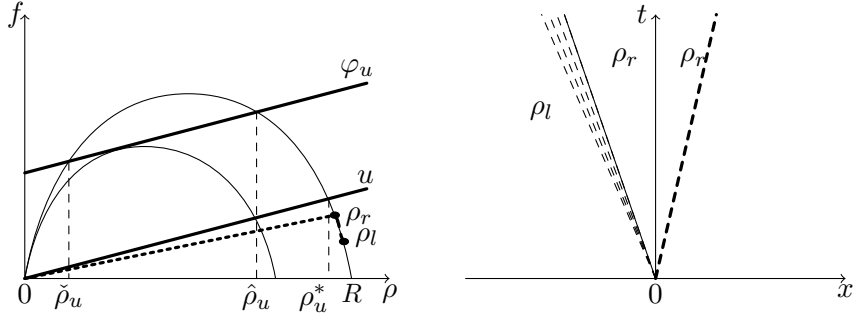

\end{enumerate}

Assume now that $u < \sigma_l$, so that
the density in front of the vehicle at $y$ is $\rho_l$.
The following possibilities hold.
\begin{enumerate}
\item $f\left(\rho_l\right) > \fhi_u\left(\rho_l\right)$.
  This hypothesis implies that $\check \rho_u < \rho_l < \hat \rho_u$.
  If $\rho_r \ge \check \rho_u$, then the solution for $\rho$ is given by
  \begin{equation*}
    \rho(t,x) = \left\{
      \begin{array}{ll}
        \rho_l 
        & \textrm{ if } x < \frac{f\left(\rho_l\right) 
          - f\left(\hat \rho_u\right)}{\rho_l - \hat \rho_u} t
        \vspace{.2cm}\\
        \hat \rho_u
        & \textrm{ if } \frac{f\left(\rho_l\right) 
          - f\left(\hat \rho_u\right)}{\rho_l - \hat \rho_u} t < x
          < ut
        \vspace{.2cm}\\
        \check \rho_u
        & \textrm{ if } ut < x < \frac{f\left(\rho_r\right) 
          - f\left(\check \rho_u\right)}{\rho_r - \check \rho_u} t
        \vspace{.2cm}\\
        \rho_r
        & \textrm{ if } x > \frac{f\left(\rho_r\right) 
          - f\left(\check \rho_u\right)}{\rho_r - \check \rho_u} t,
      \end{array}
    \right.
  \end{equation*}
  while if $\rho_r < \check \rho_u$, then the solution for $\rho$ is given by
  \begin{equation*}
    \rho(t,x) = \left\{
      \begin{array}{ll}
        \rho_l 
        & \textrm{ if } x < 
          \frac{f\left(\rho_l\right) - f\left(\hat \rho_u\right)}
          {\rho_l - \hat \rho_u} \, t
        \vspace{.2cm}\\
        \hat \rho_u
        & \textrm{ if } \frac{f\left(\rho_l\right) - f\left(\hat \rho_u\right)}
          {\rho_l - \hat \rho_u} \, t < x
          < ut
        \vspace{.2cm}
        \\
        \check \rho_u
        & \textrm{ if } ut < x < f'\left(\check \rho_u\right) t
        \vspace{.2cm}\\
        \left(f'\right)^{-1}\left(\frac{x}{t}\right)
        & 
          \textrm{ if }  f'\left(\check \rho_u\right) t < x
          < f'\left(\rho_r\right) t
        \vspace{.2cm}\\
        \rho_r
        & \textrm{ if } x > f'\left(\rho_r\right) t.
      \end{array}
    \right.
  \end{equation*}
  The solution for $y$ is given by
  \begin{equation*}
    y(t) = u t;
  \end{equation*}
  see Figure~\ref{fig:case-rarefaction-a-2}.
  Finally note that $u \le v\left(\check \rho_u\right)$,
  so that the constraint in~(\ref{eq:RP}) for $u$ is satisfied.
  \begin{figure}
    \centering
    \begin{tikzpicture}[line cap=round,line join=round,x=1.cm,y=0.7cm]

      \draw[<->] (0.2, 6.) -- (0.2, 1.) -- (5, 1.);

      \draw (0.2, 1.) to [out = 80, in = 180] (2.35, 4.5) to [out = 0,
      in = 100] (4.5, 1.);

      \draw (0.2, 1.) to [out = 80, in = 180] (1.77, 3.5) to [out = 0,
      in = 100] (3.5, 1.);

      \draw[very thick] (0.2, 1.) -- (0.2 + 4.5,
      1. + 1.7);
      \draw[very thick] (0.2, 3) --
      (0.2 + 4.5, 3. + 1.7);

      \draw[dashed] (4.2, 2.5) -- (4.2, 1.); \draw[dashed] (3.25,
      4.17) -- (3.25, 1.); \draw[dashed] (0.78, 3.2) -- (0.78, 1.);

      \node[anchor=north,inner sep=0] at (0.2,0.9) {$0$};
      \node[anchor=north,inner sep=0] at (4.,0.9) {$\rho_u^*$};
      \node[anchor=north,inner sep=0] at (3.25,0.9) {$\hat\rho_u$};
      \node[anchor=north,inner sep=0] at (0.78,0.9) {$\check\rho_u$};
      \node[anchor=north,inner sep=0] at (5,0.9) {$\rho$};
      \node[anchor=east,inner sep=0] at (0.2,6) {$f$};
      \node[anchor=north,inner sep=0] at (4.5,0.9) {\small $R$};
      \node[anchor=south,inner sep=0] at (0.4,2.5) {$\rho_r$};
      \node[anchor=south,inner sep=0] at (1.3,4.2) {$\rho_l$};

      \node[anchor=south east,inner sep=0] at (4.8,4.8) {$\fhi_u$};
      \node[anchor=south east,inner sep=0] at (4.8,2.8) {$u$};

      \draw[fill] (1.3, 4.) circle (0.06);
      \draw[fill] (0.6, 2.8) circle (0.06);
      \draw[very thick, dashed] (1.3, 4.) to [out=220, in=60]
      (0.6, 2.8);

      \draw[->] (6,1) -- (11, 1);
      \draw[->] (8.5, 1) -- (8.5, 6);

      \draw[color=black] (8.5, 1) -- (9, 6);

      \draw[very thick] (8.5, 1) -- (9.6, 6); 
      \draw[dashed] (8.5, 1) -- (10.1, 6);
      \draw[dashed] (8.5, 1) -- (10.2, 6);
      \draw[dashed] (8.5, 1) -- (10.3, 6);
      \draw[dashed] (8.5, 1) -- (10.4, 6);

      \node[anchor=south,inner sep=0] at (7,4.) {$\rho_l$};
      \node[anchor=south,inner sep=0] at (10,4.) {$\rho_r$};

      \node[anchor=south,inner sep=0] at (9.2,5.)  {$\hat \rho_u$};

      \node[anchor=south,inner sep=0] at (9.65,5.) {$\check \rho_u$};

      \node[anchor=north,inner sep=0] at (11.,0.9) {$x$}; 
      \node[anchor=east,inner sep=0] at (8.4,6) {$t$}; 
      \node[anchor=north,inner sep=0] at (8.5,0.9) {$0$};

    \end{tikzpicture}
    \caption{The Riemann problem.  The situation of a rarefaction with speed
      greater than $u$ and $f\left(\rho_l\right) > \fhi_u(\rho_l)$. In
      this case $y(t) = y(0) + ut$, while the solution for $\rho$ is
      composed by two classical waves and one non classical shock.}
    \label{fig:case-rarefaction-a-2}
  \end{figure}
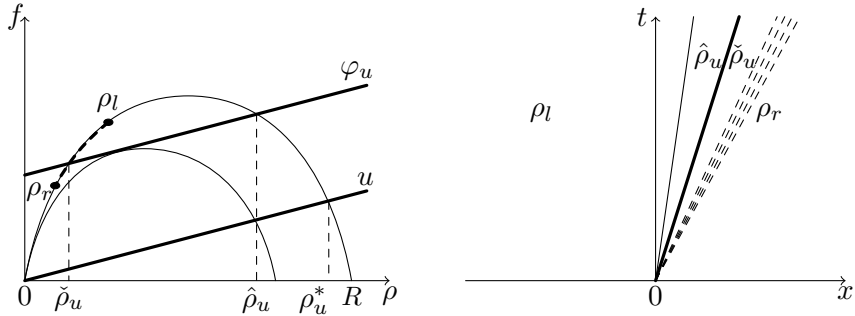

\item $f\left(\rho_l\right) \le \fhi_u\left(\rho_l\right)$.
  Since $u < \sigma_l < \sigma_r$,
  this hypothesis implies that
  $\rho_r < \rho_l \le \check \rho_u$.
  In this case the solution is given by
  \begin{equation*}
    \rho(t,x) = \left\{
      \begin{array}{ll}
        \rho_l 
        & \textrm{ if } x < f'\left(\rho_l\right) t
        \vspace{.2cm}\\
        \left(f'\right)^{-1} \left(\frac{x}{t}\right)
        &
          \textrm{ if } f'\left(\rho_l\right) t < x < f'\left(\rho_r\right) t 
        \vspace{.2cm}\\
        \rho_r
        & \textrm{ if } x > f'\left(\rho_r\right) t
      \end{array}
    \right.
  \end{equation*}
  and
  \begin{equation*}
    y(t) = u t;
  \end{equation*}
  see Figure~\ref{fig:case-rarefaction-b-2}.
  Finally note that $u \le v\left(\rho_l\right)$,
  so that the constraint in~(\ref{eq:RP}) for $u$ is satisfied.
  \begin{figure}
    \centering
    \begin{tikzpicture}[line cap=round,line join=round,x=1.cm,y=0.7cm]
      
      \draw[<->] (0.2, 6.) -- (0.2, 1.) -- (5, 1.);

      \draw (0.2, 1.) to [out = 80, in = 180] (2.35, 4.5) to [out = 0,
      in = 100] (4.5, 1.);

      \draw (0.2, 1.) to [out = 80, in = 180] (1.77, 3.5) to [out = 0,
      in = 100] (3.5, 1.);

      \draw[very thick] (0.2, 1.) -- (0.2 + 4.5,
      1. + 1.7);
      \draw[very thick] (0.2, 3) --
      (0.2 + 4.5, 3. + 1.7);

      \draw[dashed] (4.2, 2.5) -- (4.2, 1.); \draw[dashed] (3.25,
      4.17) -- (3.25, 1.); \draw[dashed] (0.78, 3.2) -- (0.78, 1.);

      \node[anchor=north,inner sep=0] at (0.2,0.9) {$0$};
      \node[anchor=north,inner sep=0] at (4.,0.9) {$\rho_u^*$};
      \node[anchor=north,inner sep=0] at (3.25,0.9) {$\hat\rho_u$};
      \node[anchor=north,inner sep=0] at (0.78,0.9) {$\check\rho_u$};
      \node[anchor=north,inner sep=0] at (5,0.9) {$\rho$};
      \node[anchor=east,inner sep=0] at (0.2,6) {$f$};
      \node[anchor=north,inner sep=0] at (4.5,0.9) {\small $R$};
      \node[anchor=south,inner sep=0] at (0.25,2.3) {$\rho_r$};
      \node[anchor=south,inner sep=0] at (0.7,3.3) {$\rho_l$};

      \node[anchor=south east,inner sep=0] at (4.8,4.8) {$\fhi_u$};
      \node[anchor=south east,inner sep=0] at (4.8,2.8) {$u$};

      \draw[fill] (0.65, 2.9) circle (0.06);
      \draw[fill] (0.5, 2.5) circle (0.06);
      \draw[very thick, dashed] (0.65, 2.9) to [out=220, in=60]
      (0.5, 2.5);

      \draw[->] (6,1) -- (11, 1);
      \draw[->] (8.5, 1) -- (8.5, 6);

      \draw[dashed] (8.5, 1) -- (10, 6);
      \draw[dashed] (8.5, 1) -- (10.1, 6);
      \draw[dashed] (8.5, 1) -- (10.2, 6);
      \draw[dashed] (8.5, 1) -- (10.3, 6);

      \draw[very thick, dashed] (8.5, 1) -- (9.6, 6); 

      \node[anchor=south,inner sep=0] at (10,4.) {$\rho_r$};

      \node[anchor=south,inner sep=0] at (9.7,5.5)  {$\rho_l$};
      \node[anchor=south,inner sep=0] at (8.1,5.)  {$\rho_l$};

      \node[anchor=north,inner sep=0] at (11.,0.9) {$x$}; 
      \node[anchor=east,inner sep=0] at (8.4,6) {$t$}; 
      \node[anchor=north,inner sep=0] at (8.5,0.9) {$0$};

    \end{tikzpicture}
    \caption{The Riemann problem.  The situation of a rarefaction with speed
      greater
      than $u$ and $f\left(\rho_l\right) \le \fhi_u(\rho_l)$.
      In
      this case $y(t) = ut$, while the solution for $\rho$ is
      composed by a classical rarefaction wave.}
    \label{fig:case-rarefaction-b-2}
  \end{figure}

\end{enumerate}

Assume finally that $\sigma_l \le u \le \sigma_r$, so that
the density in front of the vehicle at $y$ is
$\tilde \rho = \left(f'\right)^{-1}(u)$. Note that
$\check \rho_u < \tilde \rho < \hat \rho_u$.
The following possibilities hold.

\begin{enumerate}
\item $\rho_l \le \hat \rho_u$ and $\rho_r \ge \check \rho_u$.
  The solution for $\rho$ is given by
  \begin{equation*}
    \rho(t,x) = \left\{
      \begin{array}{ll}
        \rho_l 
        & \textrm{ if } x < \frac{f\left(\rho_l\right) 
          - f\left(\hat \rho_u\right)}{\rho_l - \hat \rho_u} t
        \vspace{.2cm}\\
        \hat \rho_u
        & \textrm{ if } \frac{f\left(\rho_l\right) 
          - f\left(\hat \rho_u\right)}{\rho_l - \hat \rho_u} t < x
          < ut
        \vspace{.2cm}\\
        \check \rho_u
        & \textrm{ if } ut < x < \frac{f\left(\rho_r\right) 
          - f\left(\check \rho_u\right)}{\rho_r - \check \rho_u} t
        \vspace{.2cm}\\
        \rho_r
        & \textrm{ if } x > \frac{f\left(\rho_r\right) 
          - f\left(\check \rho_u\right)}{\rho_r - \check \rho_u} t,
      \end{array}
    \right.
  \end{equation*}
  while the solution for $y$ is
  \begin{equation*}
    y(t) = u t.
  \end{equation*}

\item $\rho_l > \hat \rho_u$ and $\rho_r \ge \check \rho_u$.
  The solution for $\rho$ is given by
  \begin{equation*}
    \rho(t,x) = \left\{
      \begin{array}{ll}
        \rho_l 
        & \textrm{ if } x < f'\left(\rho_l\right) t
        \vspace{.2cm}\\
        \left(f'\right)^{-1}\left(\frac{x}{t}\right)
        & \textrm{ if } f'\left(\rho_l\right) t < x < 
          f'\left(\hat \rho_u\right) t
        \vspace{.2cm}\\
        \hat \rho_u
        & \textrm{ if } f'\left(\hat \rho_u\right) t < x
          < ut
        \vspace{.2cm}\\
        \check \rho_u
        & \textrm{ if } ut < x < \frac{f\left(\rho_r\right) 
          - f\left(\check \rho_u\right)}{\rho_r - \check \rho_u} t
        \vspace{.2cm}\\
        \rho_r
        & \textrm{ if } x > \frac{f\left(\rho_r\right) 
          - f\left(\check \rho_u\right)}{\rho_r - \check \rho_u} t,
      \end{array}
    \right.
  \end{equation*}
  while the solution for $y$ is
  \begin{equation*}
    y(t) = u t.
  \end{equation*}

\item $\rho_l \le \hat \rho_u$ and $\rho_r < \check \rho_u$.
  The solution for $\rho$ is given by
  \begin{equation*}
    \rho(t,x) = \left\{
      \begin{array}{ll}
        \rho_l 
        & \textrm{ if } x < \frac{f\left(\rho_l\right) 
          - f\left(\hat \rho_u\right)}{\rho_l - \hat \rho_u} t
        \vspace{.2cm}\\
        \hat \rho_u
        & \textrm{ if } \frac{f\left(\rho_l\right) 
          - f\left(\hat \rho_u\right)}{\rho_l - \hat \rho_u} t < x
          < ut
        \vspace{.2cm}\\
        \check \rho_u
        & \textrm{ if } ut < x < f'\left(\check \rho_u\right)t
        \vspace{.2cm}\\
        \left(f'\right)^{-1}\left(\frac{x}{t}\right)
        & \textrm{ if } f'\left(\check \rho_u\right)t < x < 
          f'\left(\rho_r\right)t
        \vspace{.2cm}\\
        \rho_r
        & \textrm{ if } x > f'\left(\rho_r\right)t,
      \end{array}
    \right.
  \end{equation*}
  while the solution for $y$ is
  \begin{equation*}
    y(t) = u t.
  \end{equation*}

\item $\rho_l > \hat \rho_u$ and $\rho_r < \check \rho_u$.
  The solution for $\rho$ is given by
  \begin{equation*}
    \rho(t,x) = \left\{
      \begin{array}{ll}
        \rho_l 
        & \textrm{ if } x < f'\left(\rho_l\right) t
        \vspace{.2cm}\\
        \left(f'\right)^{-1}\left(\frac{x}{t}\right)
        & \textrm{ if } f'\left(\rho_l\right) t < x < 
          f'\left(\hat \rho_u\right) t
        \vspace{.2cm}\\
        \hat \rho_u
        & \textrm{ if } f'\left(\hat \rho_u\right) t < x
          < ut
        \vspace{.2cm}\\
        \check \rho_u
        & \textrm{ if } ut < x < f'\left(\check \rho_u\right)t
        \vspace{.2cm}\\
        \left(f'\right)^{-1}\left(\frac{x}{t}\right)
        & \textrm{ if } f'\left(\check \rho_u\right)t < x < 
          f'\left(\rho_r\right)t
        \vspace{.2cm}\\
        \rho_r
        & \textrm{ if } x > f'\left(\rho_r\right)t,
      \end{array}
    \right.
  \end{equation*}
  while the solution for $y$ is
  \begin{equation*}
    y(t) = u t.
  \end{equation*}

\end{enumerate}

\end{appendices}

%
%

\bibliographystyle{abbrv}
    
\bibliography{biblio,refs}

\end{document}